\documentclass[oneside]{amsart}

\usepackage{style} 

\title[Transport Inequalities for Carnot Path Spaces]{Talagrand-Type Transport Inequalities for Path Spaces over  Carnot Groups}

\author{Peter K.\ Friz$^{1,2}$}
\author{Helena Kremp$^{1,2}$}
\author{Vaios Laschos$^2$}
\author{Matthias Liero$^2$}
\author{Benjamin A.\ Robinson$^3$}
\address{$^1$Technische Universtit\"at Berlin, Berlin, Germany}
\address{$^2$Weierstra{\ss}-Institut f\"ur Angewandte Analysis und Stochastik, Berlin, Germany}
\address{$^3$Department of Statistics, University of Klagenfurt, Austria}

\email{friz@math.tu-berlin.de} \email{kremp@wias-berlin.de} \email{artnoage@gmail.com} \email{liero@wias.berlin.de} \email{benjamin.robinson@aau.at}



\date{\today}

\begin{document}

\maketitle

\begin{abstract}
We consider Talagrand-type transportation inequalities for the law of Brownian motion on Carnot groups. An important example is the lift of standard Brownian motion to the Brownian rough path.  
We present a direct proof on enhanced path space, which also yields equality when restricting to adapted couplings in the transport problem. Moreover, we prove a Talagrand inequality for the heat kernel measure on Carnot groups and deduce the inequality for the law of Brownian motion on Carnot groups via a bottom-up argument.
Our study of this enhanced Wiener measure contributes to a longstanding programme to extend key properties of Wiener measure to the non-commutative setting of the enhanced Wiener measure, which is of central importance in Lyons' rough path theory.
With a non-commutative sub-Riemannian  state space,  we observe phenomena that differ from the Euclidean case. 
In particular, while a top-down projection argument recovers Talagrand's inequality on Euclidean space from the corresponding inequality on the path space, such a projection argument breaks down in the Carnot group setting.
We further study a Riemannian approximation of the Heisenberg group, in which case the failure of the top-down projection can be partially overcome. Finally, we show that the cost function used in the Talagrand inequality is a natural choice, in that it arises as a limit of discretised costs in the sense of $\Gamma$-convergence.
\end{abstract}


\tableofcontents

\section{Introduction}
\noindent
Let $\mu \in \mathcal{P} (E)$ be a Borel probability measure on a Polish space $E$. Given a measurable cost function $c\colon E \times
E \to [0, \infty]$, we say that $\mu$ satisfies Talagrand's $\mathcal{T}_2$ transport inequality with constant $\alpha > 0$, and
write
$ \mu \in \mathcal{T}_2 (E, c, \alpha)$,
if for every $\nu \in \mathcal{P} (E)$ it holds that
\[ \T_{c, 2}^2(\mu, \nu) \coloneqq \inf_{\lambda \in \Pi (\mu, \nu)}  \int_{E \times E} c^2 (x, y) 
   \hspace{0.17em} \D \lambda (x, y) \hspace{0.27em} \le \hspace{0.27em}
   \frac{2}{\alpha}  \hspace{0.17em} H (\nu \| \mu), \]
where $\Pi (\mu, \nu)$ denotes the set of couplings between $\mu$ and $\nu$
and $H (\nu \| \mu)$ is the relative entropy of $\nu$ with respect to $\mu$.
If the cost is induced by a metric $d$ on $E$, that is $c (x, y) = d (x, y)$,
the above definition reduces to the classical $2$-Wasserstein formulation,
\[ W_2^2 (\mu, \nu) \hspace{0.27em} \le \hspace{0.27em} \frac{2}{\alpha} 
   \hspace{0.17em} H (\nu \| \mu), \qquad \forall\nu \in \mathcal{P} (E) . \]

We note, without going into details, that there is an important connection to concentration of measure and the log-Sobolev inequalities by results from \cite{OtVi00}. Talagrand \cite{Ta96a} first proved a $\Tc_2$ inequality for the standard Gaussian measure on $\R^d$ with Euclidean cost. 
A $\Tc_2$ inequality for $\R^d$-valued Brownian motion with a cost given in terms of the Cameron--Martin distance first appeared in \cite{FeUe02}. Later
\cite{Le13} gave a similar proof, using the intrinsic drift from \cite{Fo86,Fo88} and Girsanov's theorem to prove the $\Tc_2$ inequality directly on Wiener space.
Alternative proofs using Girsanov's theorem also appeared in \cite{DjGuWu04} and \cite{FeUe04}.
On the other hand, 
the $\Tc_2$ inequality on Wiener space can also be derived as a consequence of the Gaussian product case.
In fact, \cite{Ta96a} already considered the infinite Gaussian product case; cf.\ \cite{riedel_transportationcost_2017} and reference therein for explicit constructions.
This so-called \emph{bottom-up} approach uses the tensorisation property of the $\Tc_2$ inequality and a truncated expansion of the Brownian motion.
As observed in \cite{Le13} and \cite{Foe21}, one can also recover Talagrand's $\Tc_2$ inequality on $\R^d$ from the $\Tc_2$ inequality on path space by considering a Brownian bridge. This gives a so-called \emph{top-down} approach to Talagrand's $\Tc_2$ inequality.
In this paper, a first study connecting aspects of optimal transport with rough analysis, 
we investigate the validity of the $\Tc_2$ inequality, as well as the bottom-up and top-down approaches, when $\R^d$ is replaced by a certain Carnot group.

The advent of rough path theory (see, e.g.\ \cite{Lyons1998, FrVi10}) has highlighted the fundamental importance of ($d$-dimensional) Brownian motion $B$ lifted to the free step-$2$ nilpotent group (over $\R^d$), which is an example of Brownian motion with values in a step-$2$ Carnot group $\mathbb{G}$. Denoted by
\[ 
\bB_t = \biggl( B_t, \; \mathrm{Anti} \Bigl( \int_0^t B_s \otimes
   \D B_s \Bigr) \biggr), 
\]
this process is also known as horizontal Brownian motion, enhanced Brownian motion, or Brownian rough path, depending on authors and context.
When $d=2$, the relevant group is nothing but the classical $(2+1)$-dimensional Heisenberg group $\mathbb{H} \cong \R^3$ with group law
\[
       \bigl( (x,y,z), (x',y',z') \bigr) \mapsto \bigl(x+x',y+y',z+z'+(xy'-x'y)/2\bigr).   
\]
Though not directly related to this work, we note that the 
 interplay of optimal transport and Heisenberg groups was pioneered in \cite{AmRi04}; see also \cite{AS2019HEFC} for recent work in the context of Carnot groups.

Let us agree on some notation. Unless otherwise stated, $\mu = \Law (B)$ denotes Wiener measure on $\Omega=C_0([0,T],\R^d)$, with Gaussian unit time marginal $\mu_1 = \mathcal{N}(0,I_d)$. Similarly, call $\bmu = \Law(\bB)$ the \emph{enhanced Wiener measure} on $$\bOmega_{\bbG} = C_0([0,T],\bbG),$$ with (non-Gaussian) unit time marginal $\bmu_1$, which we call the \emph{heat kernel measure} on $\bbG$.  Over the last 20 years, starting with \cite{LQZ02}, numerous properties of Wiener measure (including sample path regularity, Cameron--Martin shifts, Schilder's large deviations, Stroock--Varadhan support theorem) have been extended from $\mu$ to $\bmu$, with significant benefits to stochastic analysis (see, e.g.\ \cite{FrVi10, LyonsICM, FH20} and references therein). See also \cite{chiusole2026abstractwienermodelspaces} for an abstract view.
This naturally raises the question of whether Talagrand's $\Tc_2$ inequality for Gaussian measures (respectively, Wiener measure) extends to heat kernel measures (respectively, enhanced Wiener measure) on $\bbG$, and to what extent the bottom-up and top-down approaches remain valid. In this article, we provide a reasonably complete answer to these questions.

We prove the $\Tc_2$ inequality for $\bmu$ with the cost function $C_{\mathcal{H}}$ on $\bOmega_\bbG$,
$$
C_\cH(\bomega, \ol \bomega) \coloneqq
	\begin{cases}
		\|h\|_\cH, & \text{if} \; \ol \bomega = T_h \bomega, \; \text{for some} \; h \in \cH,\\
        + \infty, & \text{otherwise},
        \end{cases}
$$
where $\cH$ is the Cameron--Martin space of $\mu$ and $T_h$ is (essentially\footnote{Contrary to the standard rough path setting, we deal here with general step-$2$ Carnot groups.}) the translation (or shift) operator known from rough path theory \cite{FrVi10}. 
We give multiple strategies to prove the $\Tc_2$ inequality for our cost $C_{\mathcal{H}}$ on $\bOmega_\bbG$, offering both a bottom-up strategy, as well as a direct approach via an application of a contraction principle \cite{DjGuWu04, riedel_transportationcost_2017} or a lifting of the result from \cite{Le13} to the Carnot group setting.

The bottom-up approach consists of showing a $\Tc_2$ inequality for the heat kernel measure $\bmu_1$ on the Carnot group and inferring the result for the enhanced Wiener measure $\bmu$ by using the tensorisation property of the Talagrand inequality.
Our approach is to discretise the Brownian motion in time, rather than to consider an expansion as in \cite{riedel_transportationcost_2017,Foe21}.
By an Otto--Villani argument \cite{OtVi00, GiLe13}, the $\Tc_2$ inequality for the heat kernel measure on a Carnot group follows from a log-Sobolev inequality. The latter is only partially available: from \cite{Li06, El10} we have certain heat semigroup estimates on the Heisenberg (and so-called H-type) 
groups, which imply the required log-Sobolev inequalities.
Given the correct heat semigroup estimate, our proof does not rely on an H-type setting and holds true for general step-$2$ Carnot groups.

We further prove the $\Tc_2$ inequality for $\bmu$ directly on the path space, via two different approaches. First, we apply a contraction principle to the lift of a standard Brownian motion to deduce the result from the $\Tc_2$ inequality for Wiener measure. This approach also extends to the lifts of more general Gaussian processes; cf.\ \cite[Chapter 15]{FrVi10}. Alternatively, we exploit F\"ollmer's intrinsic drift from \cite{Fo86,Fo88} to prove the $\Tc_2$ inequality for $\bmu$, following the strategy of \cite{Le13}.
The latter approach gives additional insights into so-called \emph{adapted} transport inequalities.
As noted in \cite{Al81,La18,BaBaBeEd19a,BaBaBeEd19b}, in the case of optimal transport problems involving laws of stochastic processes, it is desirable to consider \textit{adapted} couplings rather than general couplings between the laws. We identify the optimal adapted transport plan for the cost $C_{\mathcal{H}}$ and show that equality holds in the $\Tc_2$ inequality when restricting to adapted couplings.

A key difference from the Euclidean setting is that the top-down approach fails in the non-commutative sub-Riemannian setting of Carnot groups. Indeed, in the Euclidean setting, given the $\Tc_2$ inequality for Wiener measure $\mu$, a contraction principle can be applied to deduce the $\Tc_2$ inequality for $\mu_1$. However, for the enhanced Wiener measure $\bmu$, this contraction principle argument breaks down and we cannot deduce the $\Tc_2$ inequality for the heat kernel measure $\bmu_1$ from the corresponding inequality for $\bmu$.
Considering a Riemannian approximation (still non-commutative) to the Carnot group, we find that the validity of the top-down approach is partially recovered. Given the $\Tc_2$ inequality for the law $\bmu^\varepsilon$ of Brownian motion on the approximating Riemannian manifold, the contraction principle implies that $\bmu^\varepsilon_1$ satisfies a $\Tc_p$ inequality for $p \in [1, 2)$, but not for $p = 2$.

We remark that our cost $C_{\mathcal{H}}$ differs from the one considered in \cite[Corollary 1.4]{riedel_transportationcost_2017}, which is defined in terms of the Cameron--Martin norm of the difference of the path in the group projected onto its first component, and which turns out to be suboptimal (see the discussion in \cref{sec:sharp-cost}).
Our cost $C_{\mathcal{H}}$ is a natural choice in the following sense: $C_{\mathcal{H}}$ can be obtained as the 
variational limit (more precisely, the $\Gamma$-limit; see Section~\ref{sec:gamma-conv}) 
of ``finite-dimensional costs'' $C_n$ that arise in our bottom-up approach:
\begin{align*}
    C_n^2(\bomega, \ol \bomega) & = 2^n \sum_{k = 1}^{2^n} d_\cc^2(\bomega_{t_{k - 1}^n, t_k^n}, \ol \bomega_{t_{k - 1}^n, t_k^n}), \quad \bomega,\ol\bomega\in\bOmega_{\bbG},
\end{align*}
where $d_\cc$ denotes the Carnot--Caratheodory metric on $\bbG$. 
We prove in Section~\ref{sec:gamma-conv} that the $\Gamma$-convergence of the cost functions $C_n$ also leads to the 
$\Gamma$-convergence of the optimal transport costs $\T_{C_n, 2}(\bmu, \cdot)$ to $\T_{C_\cH, 2}(\bmu, \cdot)$.

Note that, while our direct approach to proving the $\Tc_2$ inequality gives an elegant and short proof that holds in greater generality, the bottom-up approach, and in particular the $\Gamma$-convergence, yields clear information on the choice of the most natural cost function.

\medskip
For the reader's convenience we summarise our findings as concise statements.

\begin{theorem}[{\em Direct approach}, cf.\ Theorem \ref{thm:t2-top-foellmer}, and extensions in \cref{sec:top-riedel}] \label{thm: main}
	The measure $\bmu$ on $\bOmega_\bbG$ satisfies the $\Tc_2$ inequality $\bmu \in \Tc_2(\bOmega_\bbG, C_\cH, 1)$.    
\end{theorem}

\begin{theorem}[{\em Bottom-up}, cf.\ Theorem \ref{thm:t2-carnot-path}]
	Suppose that there exists $\alpha > 0$ such that $\bmu_1 \in \Tc_2(\bbG, d_\cc, \alpha)$. Then $\bmu \in \Tc_2(\bOmega_\bbG, C_\cH, \alpha)$.
\end{theorem}

\begin{theorem}[{\em $\mathcal{T}_2$ on group}, cf.\ Theorem \ref{thm:t2-h-type}]
	Let $\bbG$ be an H-type group. Then there exists $\alpha > 0$ such that $\bmu_1 \in \Tc_2(\bbG, d_\cc, \alpha)$.
\end{theorem}

\begin{theorem}[{\em Adapted couplings}, cf.\ Theorem \ref{thm:t2-top-foellmer}]\label{thm: main2}
	Let $\bnu$ be a probability measure on $\bOmega_\bbG$ with $\bnu \ll \bmu$. Then the optimal adapted coupling between $\bmu$ and $\bnu$ is given explicitly and
	\begin{align*}
		\T_{C_{\mathcal{H}}, 2}^{\ad}(\bmu,\bnu)^2 = 2H(\bnu \| \bmu).
	\end{align*}
\end{theorem}

\begin{theorem}[{\em Cost approximation}, cf.\ \cref{cor:gamma-convergence,thm:GammaConvTransport}]\label{thm: main3}
~
\begin{itemize}
    \item[(i)] Pointwise convergence $C_{n} \to  C_{\mathcal{H}}$ fails (by example),
    \item[(ii)] $\Gamma$-convergence\footnote{Recall that $\Gamma$-convergence is
a natural notion of convergence from the theory of calculus of variations
for sequences of functionals, which guarantees the convergence of minimisers and minima.} $C_{n}\stackrel{\Gamma}{\to} C_{\mathcal{H}}$ holds with respect to the uniform topology on $\bOmega_\bbG \times \bOmega_\bbG$, 
    \item[(iii)] $\T_{C_n, 2}(\bmu, \cdot) \xrightarrow{\Gamma} \T_{C_\cH, 2}(\bmu, \cdot)$ with respect to the weak topology on $\Pc(\bOmega_\bbG)$.
\end{itemize}
\end{theorem}

\begin{theorem}[{\em Top-down -- validity vs.\ failure}]A contraction principle
\begin{itemize}
    \item[(i)] gives the implication 
    $\mu \in \Tc_2(\Omega, \| . - . \|_\cH, \alpha)
    \implies \mu_1 \in \Tc_2(\R^d, |.-.|, \alpha)$,    
    \item[(ii)] does not give 
    $\bmu \in \Tc_2(\bOmega, C_\cH, \alpha)
    \implies 
    \bmu_1 \in
    \Tc_2(\bbG, d_\cc, \alpha')
   $, no matter 
   $\alpha,\alpha'>0,$
   \item[(iii)] gives a weak implication $\bmu^{\epsilon} \in \Tc_{2}(\bOmega_{\bbH_{\epsilon}},C^{\epsilon}_{\cH},\alpha) 
    \implies 
    \bmu^\varepsilon_1 \in
    \Tc_{p}(\bbH_{\epsilon},d_{\epsilon},\wt\alpha(\varepsilon, p))
   $, for $p \in [1, 2)$, where $(\bbH_{\epsilon},d_{\epsilon})$ is a Riemannian approximation to the $(2+1)$-dimensional Heisenberg group. 
\end{itemize}
\end{theorem}

The paper is structured as follows. \cref{sec:setting} contains preliminary results on Talagrand inequalities and introduces our setting of step-$2$ Carnot groups.
In \cref{sec:top-foellmer}, we present a direct approach to proving the $\Tc_2$ inequality for the law of Brownian motion on a Carnot group via F\"ollmer's intrinsic drift, and we show equality for the case of adapted transport plans.
In \cref{sec:top-riedel}, we prove the $\Tc_2$ inequality for general Gaussian rough paths by a direct approach using a contraction principle.
\cref{sec:heat-kernel} studies the Talagrand inequality for the heat kernel measure $\bmu_1$ on Carnot groups, as well as its connection to log-Sobolev inequalities and heat semigroup estimates. \cref{ss:limitPassageTalagrand} presents a bottom-up approach to proving the $\Tc_2$ inequality for $\bmu$ as a consequence of the results of \cref{sec:heat-kernel}.
In \cref{sec:projection-blow-up}, we show by example that we cannot project the $\Tc_2$ inequality for $\bmu$ down to a $\Tc_2$ inequality for $\bmu_1$, and that the cost functions $C_n$ blow up pointwise.
In \cref{sec:riemann}, we study a Riemannian approximation of the Heisenberg group, for which we can partially overcome the issues of the cost blow-up and failure of projection.
In \cref{sec:gamma-conv}, we prove the $\Gamma$-convergence of the costs $C_n$ to $C_{\mathcal{H}}$.
We conclude in \cref{sec:outlook} by commenting on the extension of our results to higher order Carnot groups.

\section{Setting}\label{sec:setting}
In this section, we first collect relevant definitions and results related to Talagrand inequalities. Next, we define our setting of step-$2$ Carnot groups and introduce a Brownian motion with paths in a Carnot group, as well as the lift and shift operation on paths. Using the shift operator, we define a suitable cost function $C_{\mathcal{H}}$, which appears in our $\Tc_2$ inequality.

\subsection{Preliminaries on Talagrand inequalities}\label{sec:prelim-t2}
Given two Borel  probability measures $\mu,\nu$ on a Polish space $E$, let $\Pi(\mu,\nu)$ denote the set of probability measures on $E\times E$ with marginals $\mu,\nu$. Such measures are called couplings (or transport plans). The relative entropy of \(\nu\) with respect to \(\mu\) is defined as
\[
	H(\nu \| \mu) = 
	\begin{cases} 
		\int_E \log \frac{\D\nu}{\D\mu} \di\nu & \text{if } \nu \ll \mu, \\
		+\infty & \text{otherwise}.
	\end{cases}
\]

\begin{definition}
	Let $E$ be a Polish space and let $c \colon E \times E \rightarrow [0,\infty]$ be a measurable function. We say that a Borel probability measure $\mu$ satisfies the \emph{cost-information inequality} on 
	$E$ with cost $c$, parameter $\alpha>0$, and exponent $p \in [1, \infty)$ if, for any Borel probability measure $\nu$ on $E$,
	\begin{align*}
		\T_{c, p}(\mu, \nu)  \le \sqrt{\frac{2}{\alpha} H (\nu \| \mu )}, 
		\quad
		\text{where} 
		\quad
		\T_{c, p}(\mu, \nu) \coloneqq \Bigl(\inf_{\lambda\in\Pi(\mu,\nu)} \int \int c^p(x, y) \di \lambda (x, y)\Bigr)^{\frac 1p}.
	\end{align*}
	We write $\mu \in \Tc_p (E, c,\alpha)$ and say that $\mu$ satisfies a $\Tc_p$ inequality.\footnote{Since $c^p$ is just another instance of a measurable function on $E \times E$, there is no loss of generality in taking $p=1$. However, we find this definition useful later in the paper.}
\end{definition}

In particular, we are interested in the case of $p = 2$. On $\R^d$, Talagrand \cite{Ta96a} proved that the standard Gaussian measure satisfies a $\Tc_2$ inequality with Euclidean cost. Talagrand's result has since been lifted to the Wiener measure on path space.

Let $B$ denote a standard Brownian motion on $\R^d$, let $\mu_1 = \Law(B_1)$ denote the standard Gaussian measure on $\R^d$, and let $\mu = \Law(B)$ denote the Wiener measure on $\Omega \coloneqq C_0([0, 1], \R^d)$.
The Cameron--Martin space for $\mu$ is defined as
\begin{equation}\label{eq:cameron-martin}
    \mathcal H \coloneqq \{\,h \colon [0, 1] \to \R^d \; \text{absolutely continuous} : \dot h \in L^2, \; h_0 = 0 \,\} = W^{1, 2}_0([0, 1], \R^d),
\end{equation}
and the Cameron--Martin norm $\|\cdot \|_{\mathcal{H}}$ is defined by $\|h\|_{\mathcal H}^2 = \int_0^1 |\dot h_t|^2 \di t$, for $h \in \mathcal H$. Throughout the text, for $p\in[1,\infty]$, $W^{1,p}([0,1], \R^d)$ denotes the usual Sobolev space, and 
$W^{1,p}_0([0,1], \R^d)$ the subspace such that $x_0 =0$ for $x \in W^{1,p}_0([0,1], \R^d)$.

Define the Cameron--Martin cost $c_\cH \colon \Omega \times \Omega \to [0, \infty]$ by
\begin{equation}\label{eq:BM-cost}
    c_\cH(x,y) \coloneqq
    \begin{cases}
        \|y - x\|_\cH, & y - x \in \cH,\\
        + \infty & \text{otherwise}.
    \end{cases}
\end{equation}
Then the Wiener measure $\mu$ satisfies the $\Tc_2$ inequality (cf.~\cite{FeUe02, DjGuWu04, FeUe04, Le13}):
\begin{equation}\label{eq:t2}
    \mu \in \Tc_2(\Omega, c_\cH, 1).
\end{equation}

\subsubsection*{Contraction principle}
The following contraction principle for $\Tc_p$ inequalities is a special case of \cite[Lemma 4.1]{riedel_transportationcost_2017}.
\begin{lemma}\label{lem:contraction}
    Let $(E, d)$, $(S, \rho)$ be metric spaces, with $(E, d)$ a Polish space, let $c \colon E \times E \to [0, \infty]$ and $\wt c \colon S \times S \to [0, \infty]$ be Borel-measurable functions, and let $\eta$ be a Borel probability measure on $E$. Let $\psi \colon E \to S$ and $L \colon E \to [0, \infty]$ be measurable functions such that
    \begin{equation*}
        \wt c(\psi(x), \psi(\ol x)) \leq L(x) c(x, \ol x),
    \end{equation*}
    for all $x, \ol x \in E_0$, where $E_0 \subseteq E$ satisfies $\eta(E_0) = 1$.
    
    Suppose that $\eta \in \Tc_2(E, c, \alpha)$, for some $\alpha \in (0, \infty)$. Then, for any $p \in [1, 2]$ such that $L \in L^q(\eta)$ for $q = \tfrac{2p}{2-p} \in [2, \infty]$, we have $\psi_{\sharp} \eta \in \Tc_p(S, \wt c, \alpha \|L\|_{L^q(\eta)}^{-2})$.
\end{lemma}

\begin{remark}\label{rem:upgrade-topology}
	In particular, the contraction principle in \Cref{lem:contraction} allows us to upgrade the topology used in \cref{thm: main} from the uniform topology to the $\beta$-H\"older topology for $\beta \in (\frac{1}{3},\frac{1}{2})$ (cf.\ \cite[Section 8]{FrVi10} for the definition of this topology).
	Indeed, for $E = C_0([0, 1], \bbG)$ and $\bmu \in \Pc_2(E)$ the law of Brownian motion on $\bbG$, we have that $\bmu(\wt E) = 1$, where $\wt E = C^\beta_0([0, 1], \bbG)$ for some $\beta \in (\frac{1}{3}, \frac{1}{2})$. Thus, defining $\wt c = c|_{\wt E \times \wt E}$, the result of \Cref{thm: main} that $\bmu \in \Tc_2(E, c, \alpha)$ extends to $\bmu \in \Tc_2(\wt E, \wt c, \alpha)$ by \Cref{lem:contraction}.
	
	We remark that such a direct upgrade of the topology is not observed in other settings. For example, showing that a large deviation principle can be lifted from the uniform topology to the H\"older topology is significantly more involved; see \cite[Theorem 39]{FrVi05}, whose proof is based on the inverse contraction principle for large deviations \cite[Theorem 4.2.4]{DeZe10}.
\end{remark}

\subsubsection*{Adapted \texorpdfstring{$\Tc_p$}{Tp} inequalities}
For a metric space $(S, d)$, the $p$-Wasserstein distance $\T_{d, p}$ metrises the weak topology on $\Pc_p(S)$. When elements of $S$ should be regarded as stochastic processes, however, this topology is not sufficient to capture the flow of information encoded in the filtrations associated to the processes. The \emph{adapted weak topology} and \emph{adapted Wasserstein distance} have been shown to be more suitable; see, e.g.\ \cite{Al81,La18,BaBaBeEd19a,BaBaBeEd19b}. The adapted Wasserstein distance is a special case of the adapted (also called bicausal) optimal transport problem, defined as follows.

\begin{definition}\label{def:adapted-coupling}
	Let $E$ be a Polish space and $\mu, \nu \in \Pc(C([0, 1], E))$. Let $\lambda \in \Pi(\mu, \nu)$ and let $X$, $Y$ be $C([0, 1], E)$-valued random variables with $\lambda = \Law(X, Y)$. Write $\F^X$ (resp.\ $\F^Y$) for the completion of the natural filtration of $X$ (resp.\ $Y$) with respect to $\mu$ (resp.\ $\nu$). We say that $\lambda$ is an \emph{adapted coupling} if the following conditional independence holds under $\lambda$: for all $t \in [0, 1]$,
	\begin{align*}
		& \F^Y_t \; \text{is independent of} \; \F^X_1 \; \text{given} \;  \F^X_t \quad
		\text{and} \quad \F^X_t \; \text{is independent of} \; \F^Y_1 \; \text{given} \; \F^Y_t.
	\end{align*}
	We denote the set of all such couplings by $\Pi_\ad(\mu, \nu)$.
    For a measurable function $c \colon C([0, 1], E) \times C([0, 1], E) \to [0, \infty]$, define the \emph{adapted optimal transport problem}
    \begin{align*}
        \T^\ad_{c,p}(\mu, \nu) \coloneqq \Bigl(\inf_{\lambda\in\Pi_\ad(\mu,\nu)} \int \int c^p(x, y) \di\lambda (x, y)\Bigr)^{\frac 1p}.
    \end{align*}
    We say that $\mu \in \Pc(C([0, 1], E))$ satisfies an \emph{adapted $T_p$ inequality} for some $p \in [1, \infty)$ if there exists $\alpha > 0$ such that
    \begin{align*}
    	\T^\ad_{c, p}(\mu, \nu)  \le \sqrt{\frac{2}{\alpha} H (\nu \| \mu )}.
    \end{align*}
\end{definition}

In this adapted setting, \cite[Lemma 5]{La18} and \cite[Theorem 3]{Foe21} show that Wiener measure satisfies an adapted $\Tc_2$ inequality with $\alpha = 1$ and that equality holds; i.e.
\begin{align}\label{eq:adapted-t2-equality}
	\T^\ad_{c_\cH, 2}(\mu, \nu) = \sqrt{2 H(\nu \| \mu)}.
\end{align}

\begin{remark}
	For continuous-time stochastic processes, \cite{BaBePaScZh23} give an alternative definition of the adapted optimal transport problem and adapted Wasserstein distance that has additional desirable topological properties. The value of this problem is defined such that it lies between $\T_{c,p}$ and $\T^\ad_{c,p}$. Thus, it is immediate that an adapted $\Tc_2$ inequality still holds in this setting. However, equality has not been studied in this case, and we leave this to future work, choosing to focus on the definition given in \Cref{def:adapted-coupling} in the present paper.
\end{remark}

\begin{remark}
	For discrete-time processes taking values in some Polish space $E$ with $n$ time steps, one can also consider their laws, which are probability measures on $E^n$, and define an adapted optimal transport problem analogously to \Cref{def:adapted-coupling}. In this setting, \cite[Corollary 1.8]{Park25} shows that the $\Tc_1$ inequality is equivalent to its adapted counterpart. Moreover, \cite[Corollary 1.9]{Park25} shows that, for probability measures with finite exponential moment, an adapted $\Tc_p$ inequality holds for all $p > 1$, with constant given explicitly in terms of the exponential moment and number of time steps, thus extending the results of \cite{BoVi05} to the adapted setting. For a standard Gaussian on $\R^n$, \cite[Proposition 5.10]{BaBeLiZa17} prove an adapted $\Tc_2$ inequality using a dynamic programming argument. As noted in \cite[Remark 5.11]{BaBeLiZa17}, equality cannot generally be expected in the discrete-time setting.
\end{remark}

\subsection{Step-\texorpdfstring{$2$}{2} Carnot groups}\label{sec:carnot-step-2}

Let $\mathbb{G}$ be a step-$2$ Carnot group, i.e.\ a connected, simply connected
nilpotent Lie group whose Lie algebra $\mfg$ of left-invariant vector
fields has dimension $m = d_1 + d_2$ and admits a stratification
$\mfg=\mathcal{V}_1 \oplus \mathcal{V}_2$ with $\mathcal{V}_2
= [\mathcal{V}_1, \mathcal{V}_1]$, $[\mathcal{V}_1, \mathcal{V}_2] = \{ 0 \}$.
Fix an adapted basis $(V_1, \ldots, V_m)$ such that $(V_1, \ldots, V_{d_1})$ is
a basis of $\mathcal{V}_1$. Using exponential coordinates, we can and will
identify $\mathbb{G}$ with $\mathbb{R}^m$,
\[
\mathbb{G} \ni \bx = (x_1, \ldots, x_{d_1}, x_{d_1 + 1}, \ldots, x_m) =
(x^{(1)}, x^{(2)}) \in \mathbb{R}^{d_1} \oplus \mathbb{R}^{d_2} \cong \mathbb{R}^m,
\]
with group law in Baker--Campbell--Hausdorff form,
\begin{equation}
\label{eq:BakCamHauFormula}
(\bx,\by)
\mapsto 
\bx\by=\bx+\by+\frac12[\bx,\by].
\end{equation}
It is not restrictive to assume that $V_i (0) = e_i$, the canonical basis vectors
of $\mathbb{R}^m$. For $\bx \in \bbG$, let $\ell_{\bx}\colon \bbG\to\bbG$ denote the left multiplication map defined by $\ell_{\bx}\by=\bx \by$, for $\by\in \bbG$, and let $\D \ell_{\bx}\colon T\bbG\to T\bbG$ denote its differential. By left invariance, $V_i (\bx) = \D \ell_{\bx} e_i$, $i = 1, \ldots, m$, $\bx \in \bbG$.

Endow $\mfg$ with a left-invariant metric $\langle\cdot,\cdot\rangle \colon \mfg\times\mfg\to \R$ that makes the $V_i$ orthonormal.
Define the \emph{structure constants} $\sco_{ij} \in \R^{d_2}$, for $i, j \in \{1, \dotsc, d_1\}$, by
$\sco_{ij}^k \coloneqq \langle [V_i, V_j], V_k \rangle = - \sco_{ji}^k$, for $k \in \{d_1{+}1, \dotsc, m\}$;
cf.\ \cite[Section 3.2]{BoLaUg2007SLGP}.
The group law in \eqref{eq:BakCamHauFormula} 
can then be written as
\begin{equation*}
\begin{gathered}
(\bx,\by)=((x^{(1)}, x^{(2)}), \, (y^{(1)}, y^{(2)}))
\mapsto \bx \by = \Big( x^{(1)} + y^{(1)}, \,
x^{(2)} + y^{(2)} + \frac12\sum_{i< j}\sco_{ij}(x^{(1)}_i
y^{(1)}_j-x^{(1)}_j
y^{(1)}_i)\Big).
\end{gathered}
\end{equation*}
For notational brevity, we introduce
the operator $\SCO\colon\R^{d_1\times d_1}\to \R^{d_2}$ given in terms of the structure constants by
\begin{equation*}
\SCO A= \sum_{i,j=1}^{d_1}w_{ij} A_{ij} =
\frac{1}{2}\sum_{i,j=1}^{d_1}w_{ij} (A_{ij}-A_{ji}) 
= \sum_{i<j}w_{ij} (A_{ij}-A_{ji}),\quad A\in\R^{d_1\times d_1}.
\end{equation*}
With this definition, we can rewrite the group law
for  $\bx=(x^{(1)}, x^{(2)})$, $\by=(y^{(1)}, y^{(2)})$ as 
\[
\bx\by = \Big( x^{(1)} + y^{(1)}, \,
x^{(2)} + y^{(2)} + \frac12 \SCO (x^{(1)}\otimes y^{(1)})\Big).
\]
Let $\Delta_\bbG \coloneqq \tfrac{1}{2} \sum_{i=1}^{d_1} V_i^2$ denote the sub-Laplacian on $\bbG$ and define the
horizontal gradient $\nabla_{\mathbb{G}}$ by its action
\[
\nabla_{\mathbb{G}} f \coloneqq \sum_{i = 1}^{d_1} (V_i f) V_i \in \mathcal{V}_1, \quad \text{for} \; f \colon \bbG \to \R.
\]
Let $H\mathbb{G} \subset T\mathbb{G}$ be the horizontal tangent bundle of
the group $\mathbb{G}$, i.e.\ the left-invariant sub-bundle of the tangent bundle
$T\mathbb{G}$ such that $H_{\boldsymbol{e}} \mathbb{G} = \{\, V(0) : V \in \mathcal{V}_1 \,\}$, where $\boldsymbol{e}$ is the identity element of $\bbG$. For $i \in \{1, 2\}$, define the projection operator $\pi_i \colon \bbG \to \R^{d_i}$ by
\begin{equation}\label{eq:projection}
	\pi_i (x^{(1)}, x^{(2)}) = x^{(i)}.
\end{equation} 
Dilation on $\mathbb{G}$ by a factor $s>0$ takes the form
\begin{equation}\label{eq:dilation}
\delta_{s} ((x^{(1)}, x^{(2)})) \coloneqq (s x^{(1)}, s^2 x^{(2)}).
\end{equation}
The Haar measure on $\mathbb{G}$ coincides with Lebesgue measure $\mathcal{L}^m$ on $\mathbb{R}^m$.
For measurable $E \subseteq \mathbb{R}^m$, we have
\begin{equation}\label{eq:homogeneity}
	\mathcal{L}^m (\delta_{s} E) =
s^Q \mathcal{L}^m (E),
\end{equation}
where $Q = d_1 + 2 d_2$ is called the \emph{homogeneous
dimension} of $\mathbb{G}$.

We endow $\mathbb{G}$ with the Carnot--Carath\'eodory structure induced by
$H\mathbb{G}$, as follows. An absolutely continuous curve $\gamma \colon [0, 1] \rightarrow \mathbb{G}$
is called horizontal if $\dot{\gamma}_t \in H_{\gamma_t} \mathbb{G}$ for almost every $t \in [0, 1]$.
The \emph{Carnot--Carathéodory distance} between $x, y \in \mathbb{G}$ is then defined as
\begin{align}\label{eq:dcc}
d_{\cc} (x, y) = \inf \biggl\{\, \int_0^1 | \dot{\gamma}_t | \di t :
\gamma \text{ horizontal}, \gamma_0 = x, \gamma_1 = y \,\biggr\},
\end{align}
where $|\cdot| = \sqrt{\langle\cdot,\cdot\rangle}$.

We remark that horizontal paths necessarily satisfy, for almost every $t \in [0, 1]$,
\begin{equation}\label{eq:proj1}
	\dot{\gamma}_t = \sum_{i = 1}^{d_1} V_i(\gamma_t) \langle
	\dot{\gamma}_t, V_i(\gamma_t )\rangle \eqqcolon \sum_{i = 1}^{d_1}
	V_i(\gamma_t) \dot{h}^i_t,
\end{equation}
and hence are in one-to-one correspondence with absolutely continuous
$h\in \AC([0, 1],\mathbb{R}^{d_1})$. We have that $h=\pi_1\gamma$ for the first-level projection $\pi_1$ from \cref{eq:projection}.
\begin{definition}[Canonical lift]
\label{def:canonicalLiftH}
	The \emph{canonical lift} $\Psi \colon \AC([0, 1],\mathbb{R}^{d_1}) \to C([0, 1], \bbG)$ is defined by $\Psi(h) \coloneqq \gamma$, where $\gamma \in C([0, 1], \bbG)$ and $h \in \AC([0, 1],\mathbb{R}^{d_1})$ are related by $\eqref{eq:proj1}$.
	Explicitly, we have
	\[
		\dot{\gamma}^{(1)}_t = \dot{h}_t, \quad \dot{\gamma}^{(2)}_t =
		\frac{1}{2} \sum_{i,j} \sco_{ij} h^i_t \dot{h}^j_t =
		\frac{1}{2}\SCO \big(h_t\otimes \dot{h}_t\big), \quad \text{for} \; t \in [0, 1].
	\]    
\end{definition}

By the Chow--Rashevskii theorem,  $d_\cc$ is in fact a distance, which is also left-invariant and homogeneous with respect to the dilations defined in \eqref{eq:dilation}. The metric space $(\bbG, d_\cc)$ is a Polish and geodesic space (see, e.g.\ \cite[Section 2.4]{AS2019HEFC}). We let $|\cdot|_\bbG$ denote the norm induced by $d_\cc$ on $\bbG$.
One can also equip $\bbG$ with the \emph{gauge distance} $d_g$ defined by
\begin{align*}
	d_g(x,y) = |(y^{-1}x)^{(1)}| + |(y^{-1}x)^{(2)}|^{\frac 12},
\end{align*}
for $x, y \in \bbG$. All homogeneous norms on $\bbG$ are equivalent. In particular, there exists a constant $\kappa \in (0, \infty)$ such that
\begin{align}\label{eq:gauge-distance}
	\frac1\kappa d_g(x,y) \leq d_\cc(x,y) \leq \kappa\, d_g(x,y);
\end{align}
see, e.g.\ \cite[Proposition 5.1.4]{BoLaUg2007SLGP}.

\begin{remark}[Metric derivative]\label{rem:metric-derivative}
The metric derivative of a curve $\gamma \colon [0,1]\to \bbG$ at $t \in [0, 1]$ is defined by
\begin{equation*}
	|\dot \gamma_t|_{d_\cc} \coloneqq \lim_{s\to t}\frac{d_\cc(\gamma(s),\gamma(t))}{|s-t|}.
\end{equation*}
If $\gamma$ is absolutely continuous, the metric derivative exists for almost every $t\in[0,1]$, and $|\dot \gamma_t|_{d_\cc} = |\dot \gamma_t|$; see \cite[Theorem 1.3.5]{Mont2001DBSM}.
Moreover, the metric derivative is minimal in the sense that $|\dot \gamma_t|_{d_\cc}\leq m$
for all $m\in L^1([0,1])$ with $d_\cc(\gamma(s),\gamma(t))\leq \int_s^t m(r)\di r$, $0\leq s<t\leq 1$; see  \cite[Theorem 1.1.2]{AGS}.
\end{remark}

\begin{remark}[Free step-2 Carnot groups]
The free step-$2$ nilpotent case $\mathbb{G} = \mathbb{F}^{d_1, 2}$ amounts to
$\mathbb{G} \cong \mathbb{R}^{d_1} \oplus \mathfrak{so}(d_1)$ (after identifying the exterior algebra $\wedge^2 \R^{d_1}$ with $\mathfrak{so}(d_1)$). The space $\mathfrak{so}(d_1)$ is
spanned by $\{\, e_{[i,j]} : 1 \le i < j \le d_1 \,\}$, where
$e_{[i,j]} \coloneqq \frac{1}{2}(e_i \otimes e_j - e_j \otimes e_i)$, and has dimension $d_2^\ast = d_1(d_1{-}1)/2$. 
Writing the bracket as $[e_i,e_j] = e_{[i,j]}$, the structure constants $w_{ij}^{[p,q]}$ 
reduce to Kronecker symbols.

All other step-$2$ Carnot groups can be seen as quotient groups of the free group, captured by
$d_2 \le d_2^\ast$ and the structure constants. For instance, the $(2n+1)$-dimensional Heisenberg group
$\mathbb{H}^{n} \cong \mathbb{R}^{2n} \oplus \mathbb{R}$ has $d_1 = 2n$, $d_2 = 1$, and
$\sco_{1,2} = \sco_{3,4} = \cdots = \sco_{2n-1, 2n} = 1$ (flip sign upon interchanging indices, zero otherwise).
For $d_1 = 2$, we recover the familiar example $\bbH = \bbH^1 \cong \mathbb{F}^{2,2}$. Letting $(x,y,z) \in \mathbb{H}$ denote a canonical element of $\mathbb H$, the left-invariant vector fields are given by
\[
	V_1 =\partial_{x}+\tfrac{1}{2}y\partial_{z}, \quad V_2 =\partial_{y}-\tfrac{1}{2}x\partial_{z}, \quad
	V_3 = [V_1,V_2] = \partial_{z}.
\]
\end{remark}

\begin{remark}[Heisenberg-type groups]
\label{rem:HtypeGroups}
A special class of step-$2$ Carnot groups is the class of \emph{Heisenberg-type groups}, or H-type groups for short, which enjoy additional properties.
Most importantly for us, Talagrand inequalities are known to hold on H-type groups; see Section \ref{sec:t2-h-type}.
We refer the interested reader to \cite[Chapter 18]{BoLaUg2007SLGP}.

A step-$2$ Carnot group $\bbG \cong \R^{d_1} \oplus \R^{d_2}$ is an \emph{H-type group} if, for each $z \in \R^{d_2}$, there exists a linear map $\sfJ_\bbG(z) \colon \R^{d_1} \to \R^{d_1}$ such that
\begin{equation*}
	\sfJ_\bbG({z})^2 = - |z|^2 \id
\quad\text{and}\quad 
	\big\langle \SCO (x{\otimes} y), 
    z \big\rangle_{\R^{d_2}} = \langle \sfJ_\bbG({z)} x, y \rangle_{\R^{d_1}}\quad\forall\,
    x, y \in \R^{d_1}.
\end{equation*}

Note that, necessarily, $d_1 \in 2\N$ and $d_2 \leq d_1/2$. As the name suggests, the Heisenberg group $\mathbb{H}^{n} \cong \mathbb{R}^{2n} \oplus \mathbb{R}$ is the canonical example of an H-type group, where the map $\sfJ_{\bbH^n}\colon\R^{2n} \to \R^{2n}$
is given by 
\[
	\sfJ_{\bbH^n}(z) =
	\begin{pmatrix}
	0&-zI_n\\
	z I_n& 0
	\end{pmatrix},
	\quad z \in \R.
\]
\end{remark}

\subsubsection*{Brownian motion on \texorpdfstring{$\bbG$}{G}}
Let $B = (B_t)_{t \in [0, 1]}$ be a $d_1$-dimensional Brownian motion on a filtered probability space $(\Omega, \F, (\F_t)_{t \geq 0}, \P)$. We define the Brownian motion $\bB$ on $\mathbb{G}$ as the continuous
$\mathbb{G}$-valued Markov process with generator $\Delta_\bbG$, that is obtained by solving the SDE
\begin{align}\label{eq:BM-in-G}
\D \bB_t = \sum_{i = 1}^{d_1} V^i(\bB_t) \, \D B^i_t,
\end{align}
more explicitly written as
\[
\D \bB^{(1)}_t = \D B_t, \quad
\D \bB^{(2)}_t =\frac{1}{2}\SCO (B_t\otimes \D B_t).
\]
Note that, since $B = (B_1, \ldots, B_{d_1})$ is a standard Brownian motion, there is
no difference between It\^o and Stratonovich integration here.

The Brownian motion $\bB$ takes values in the space 
\begin{equation*}
\bOmega_{\bbG} \coloneqq C_{0}([0,1],\mathbb{G})
\end{equation*} of continuous $\mathbb{G}$-valued paths started from the origin. We write $\bmu = \Law(\bB)$ and $\mu = \Law(B)$, and $\bmu_t = \Law(\bB_t)$, $\mu_t = \Law(B_t)$, for $t > 0$. For a path $\bomega \in \bOmega_\bbG$, let $\bomega_{s,t} = \bomega_s^{-1}\bomega_t$, $s \leq t$, denote its increments.

We equip the space $\bOmega_\bbG$ with the uniform topology induced by the metric $d_\infty$ defined by
\begin{align}\label{eq:unif-metric}
    d_\infty(\bomega, \ol \bomega) \coloneqq \sup_{t \in [0, 1]}d_\cc(\bomega_t, \ol \bomega_t).
\end{align}
Note that $(\bOmega_\bbG, d_\infty)$ is a Polish space.

By H\"ormander's theorem, $\Delta_\bbG$ is a hypoelliptic operator, and so the associated heat kernel 
$\mathfrak{p} \colon (0, \infty) \times \bbG \to (0, \infty)$ is smooth \cite{H1967HSOD,K1973POH,H2011MPHT,BB2015PHTS}. Note that, for all $t > 0$, the density of $\bmu_t$ is $\mathfrak{p}_t \colon \bbG \to (0, \infty)$. We also define the heat semigroup $P_t = e^{t\Delta_\bbG}$, for $t > 0$, by
\begin{align}\label{eq:semigroup}
	P_t f (\bx) = \int_\bbG f(\bx\by^{-1}) \mathfrak{p}_t(\by) \di \by =\int_\bbG f(\by) \mathfrak{p}_t(\by^{-1} \bx) \di \by, \quad \bx \in \bbG,
\end{align}
for any $f \in L^1(\bmu_t)$, with $P_0$ equal to the identity operator.

\subsubsection*{Shifting \texorpdfstring{$\bbG$}{G}-valued paths}
Let $\AC_0([0, 1], \R^d)$ denote the space of absolutely continuous curves started from the origin, and recall the canonical lift $\Psi \colon \AC_0([0,1],\R^{d_1}) \to \bOmega_\bbG$ from \cref{def:canonicalLiftH}.
Since $(\mathbb{G}, d_{\cc})$ is a geodesic space, the following approximation lemma is immediate (cf.\ \cite[Lemma 5.19, Theorem 7.32]{FrVi10}). 

\begin{lemma}[Geodesic approximations]\label{lem:geo-approx}
Every continuous $\mathbb{G}$-valued path $\bomega$ on $[0, T]$ is the uniform limit
of absolutely continuous horizontal curves; i.e.\ 
$\bomega^n = \Psi(\omega^n)$, with $\omega^n\in\AC_0([0, 1],\mathbb{R}^{d_1})$ and $\Psi$ defined in Definition~\ref{def:canonicalLiftH}.
\end{lemma}

We now extend the canonical lift to a lift map on the space $\Omega = C_0([0, 1], \R^{d_1})$ of continuous curves started from the origin.
Note that the geodesic approximation of $\bomega = (\bomega^{(1)}, \bomega^{(2)})$ from \cref{lem:geo-approx} depends on both $\bomega^{(1)}$ and $\bomega^{(2)}$. Thus we also introduce an approximation based only on $\bomega^{(1)}$ in order to extend the canonical lift.
For a continuous path $\omega \in \Omega$, let $\wh \omega^n \in \AC_0([0,1],\R^{d_1})$ denote the piecewise-linear approximation of $\omega$ on the dyadic grid $(k 2^{-n})_{k \in \{0, \dotsc, 2^n\}}$, for $n \in \N$, and note that $\wh \omega^n \to \omega$ with respect to the uniform topology on $\Omega$.
\begin{definition}(Lift)\label{def:lift}
	Extend the canonical lift $\Psi \colon \AC_0([0, 1], \R^{d_1}) \to \bOmega_\bbG$ to the \emph{lift map} $\Psi \colon \Omega \to \bOmega_\bbG$ by	\begin{align}\label{eq:lift}
		\Psi (\omega)\coloneqq
		\begin{cases}
            \lim_{n\to\infty}\Psi(\wh{\omega}^n), &\text{if the limit exists},\\
		    0,&\text{otherwise},
		\end{cases}
	\end{align}
	for $\omega \in \Omega$. Define the \emph{domain} of $\Psi$ as $\dom(\Psi) \coloneqq \{\,\omega \in \Omega \; : \; \lim_{n \to \infty} \Psi(\wh \omega^n) \; \text{exists}\,\} \subset \Omega$.
\end{definition}
The following statement holds by a minor modification to the proof of \cite[Corollary 13.19]{FrVi10} in the general step-2 Carnot setting.

\begin{proposition}\label{prop:BM-lift}
    Let $B$ be a $d_1$-dimensional Brownian motion. Then $\boldsymbol{B}$ as defined in \cref{eq:BM-in-G} satisfies $\boldsymbol{B}=\Psi(B)=\lim_{n\to\infty}\Psi(\wh B^{n})$ almost surely. In particular, it follows that $\mu(\dom (\Psi))=1$.
\end{proposition}

Having defined the lift $\Psi$ for general curves $\omega\in C_0 ([0,1],\R^{d_1})$, we can now formulate the following lemma relating absolutely continuous measures $\bnu\ll\bmu$ on $\bOmega_\bbG$ with absolutely continuous measures $\nu \ll \mu$ on $\Omega$.

\begin{lemma}\label{lem:abs-cont-lift}
	Let $\bnu$ be a Borel probability measure on $\bOmega_\bbG$. Then $\bnu \ll \bmu$ if and only if there exists a Borel probability measure $\nu$ on $\Omega$ such that $\nu \ll \mu$ and $\bnu = \Psi_\sharp \nu$.
\end{lemma}

\begin{proof}
	First suppose that there exists $\nu \ll \mu$ with $\bnu = \Psi_\sharp \nu$. Then, for any Borel $A \subseteq \bOmega_\bbG$ with $\bmu(A) = 0$, we have that $\mu(\Psi^{-1}(A)) = \bmu(A) = 0$, and so
	\begin{align*}
		\bnu(A) = \nu(\Psi^{-1}(A)) = 0.
	\end{align*}
	
	Now suppose that $\bnu \ll \bmu$.
	Let $\brho\coloneqq\frac{d\bnu}{d\bmu}$ and define the measure $\nu$ via $\frac{d\nu}{d\mu}=\brho \circ \Psi$ so that $\nu \ll \mu$.
	By \Cref{prop:BM-lift}, $\bmu(\bOmega_\bbG \setminus \Psi(\Omega)) = 0$ and $\bmu=\Psi_{\sharp}\mu$. Therefore
    \begin{align*}
    	\nu(\Omega) = \int_{\Omega}\brho(\Psi(\omega))\mu(\D\omega) = \int_{\Psi(\Omega)}\brho(\bomega)\bmu(\D\bomega) = \int_{\bOmega_\bbG}\brho(\bomega)\bmu(\D\bomega) = 1,
    \end{align*}
    and so $\nu$ is a probability measure on $\Omega$.
	Moreover, for any Borel measurable $A\subset \bOmega_\bbG$,
    \begin{align*}
        \bnu(A)=\int_{A}\brho(\bomega)\bmu(\D \bomega) =\int_{\Psi^{-1}(A)}\brho(\Psi(\omega))\mu (\D\omega)=\int_{\Psi^{-1}(A)}\nu (\D \omega)=\nu(\Psi^{-1}(A))
    \end{align*}
    Hence $\bnu=\Psi_{\sharp}\nu$.
\end{proof}

We show that the following shift map is well defined in \Cref{prop:shift} below.

\begin{definition}(Shift map)\label{def:shift-map}
For $h\in\AC_0([0,1],\R^{d_1})$ define the shift map $T_{h}\colon\bOmega_\bbG\to\bOmega_\bbG$ by
\begin{align}\label{def:rp-shift}
T_h \bomega = \lim_{n \to \infty} \Psi(\omega^n + h), \quad \bomega\in\bOmega_\bbG,
\end{align}
where $(\omega^n)_{n \in \N}$ denotes the geodesic approximation from \cref{lem:geo-approx}.
\end{definition}

\begin{proposition}\label{prop:shift}
	The shift map defined in \cref{def:shift-map} satisfies the following:
	\begin{enumerate}[label=(\roman*)]
		\item For $h \in \cH$, the shift map $T_h$ is well defined.
		\item For $\bX=(\bX^{(1)},\bX^{(2)}) \in \bOmega_\bbG$ and $h \in \cH$ the shift map $T_h \bX$ is explicitly given by
		\begin{align*}
			(T_h \bX)^{(1)}_t = \bX^{(1)}_t + h_t, \quad
			(T_h \bX)^{(2)}_t = \bX^{(2)}_t + \frac{1}{2}\SCO
			\left( \mathbb{X}^1_t 
			+ \mathbb{X}^2_t + \mathbb{X}^3_t \right),
		\end{align*}
		where 
		\[
		\D \mathbb{X}^{1}_t = \bX^{(1)}_t\otimes \D h_t, 
		\quad
		\D \mathbb{X}^{2}_t = h_t\otimes \D\bX^{(1)}_t,
		\quad
		\D \mathbb{X}^{3}_t = h_t\otimes \D h_t.
		\]
		\item If $\boldsymbol{X}$ is given by $\boldsymbol{X}=\Psi(X)$ for $X\in \dom(\Psi)$, then, for $h \in \cH$, 
		\begin{align}\label{eq:lift-shift}
		    T_{h}\boldsymbol{X}=T_{h}\Psi(X)=\Psi(X+h).
		\end{align}
		\item The map $\AC_0([0,1],\R^{d_1})\times \bOmega_\bbG\to \bOmega_\bbG,\, (h,\boldsymbol{X})\mapsto T_{h}\boldsymbol{X}$ is continuous.
	\end{enumerate}
\end{proposition}

\begin{proof}
If $\bX = \Psi(X)$ is an absolutely continuous horizontal curve, we have
$T_h \bX = \Psi(X + h)$, for $h \in \cH$. Then (ii) follows by definition of the lift on $\AC_0([0,1],\R^{d_1})$.
The representation of (ii), with cross integrals in $\mathbb{X}^1, \mathbb{X}^2$, remains meaningful
when $X$ is only continuous, by basic properties of Riemann--Stieltjes integration. By continuity properties of Riemann--Stieltjes integration, we obtain that the limit in \eqref{def:rp-shift} exists, so that (i) and (ii) follow.

Now let $X \in \dom(\Psi)$, so that $\bX=\Psi(X)=\lim_{n\to\infty}\Psi(\wh{X}^n)$ for the piecewise linear approximation $(\wh{X}^n)$. Let $h \in \cH$. Then $\boldsymbol{X}^{(1)}=\lim_{n\to\infty} \wh X^{n}$ and $\boldsymbol{X}^{(2)}= \lim_{n\to\infty} \frac12\SCO\mathbb{X}^{0,n}$, where $\D\mathbb{X}^{0,n}_t = \wh X^n_t\otimes \D\wh X^n_t$.
By the definition of $\Psi$ on $\AC_0([0,1],\R^d)$, we have that 
\[\Psi(\wh X^n+h)=\Big(\wh X^{n}+h, \,
\frac{1}{2} \SCO(\mathbb{X}^{0,n}_t+\mathbb{X}^{1,n}_t+\mathbb{X}^{2,n}_t+\mathbb{X}^{3}_t)\Big),
\] 
with $\mathbb{X}^{1,n}$, $\mathbb{X}^{2,n}$ defined as in (ii) with $X$ replaced by $\wh X^n$, and $\mathbb{X}^{3}$ defined as in (ii). By continuity of the Riemann--Stieltjes integral and (ii), we deduce that $\lim_{n\to\infty} \Psi(\wh X^{n}+h)=T_{h}\Psi(X)$. This proves (iii).

Finally, (iv) also follows from continuity properties of the Riemann--Stieltjes integral and the representation from (ii). 
\end{proof}

\begin{remark}\label{rem:translation-regularity}
    As is plain from Proposition \ref{prop:shift}, part (ii), we can translate any $\bX=(\bX^{(1)},\bX^{(2)}) \in \bOmega_\bbG$ in the direction of any absolutely continuous $h$. The situation is more complicated when dealing with Carnot groups of level strictly greater than $2$, cf.\ \cite[Section 9.4.6]{FrVi10}, or when $h$ has less regularity, as is the case for Cameron--Martin paths of fractional Brownian motion with Hurst parameter $H < 1/2$; cf.\ \Cref{sec:top-riedel}. In these cases, one has to incorporate suitable $p$-variation or H\"older rough path regularity on the path space of $\bX$.
\end{remark}

\begin{remark}\label{rem:commutativity-error}
	For any $h \in \cH$ and $\bX \in \bOmega_\bbG$, manipulating the expression for $T_h \bX$ from \cref{prop:shift} gives $\bX_{s,t}^{-1}(T_h \bX)_{s,t} = (\bZ^{(1)}_{s, t}, \bZ^{(2)}_{s, t})$, for $s, t \in [0, 1]$, $s \leq t$, where
	\begin{align*}
		\bZ^{(1)}_{s, t}  = h_{s,t},\quad
		\bZ^{(2)}_{s, t}  = \frac{1}{2}\Big(\mathbb{X}^{1}_{s,t} + \mathbb{X}^{2}_{s,t} + \mathbb{X}^{3}_{s,t} 
        + 2X_t \otimes h_s 
        - (X_s \otimes h_s  
        + X_t \otimes h_t 
        + h_s \otimes h_t)\Big).
	\end{align*}
	Also define $\bh = \Psi(h)$. Then, after integrating by parts, we find that, for any $s, t \in [0, 1]$ with $s \leq t$,
	\begin{align*}
		\btheta_{s,t}\coloneqq\bh^{-1}_{s, t} \bX_{s,t}^{-1}(T_h \bX)_{s,t} = \Big(0, \int_s^t \SCO h_{s, r} \otimes \D X_r\Big).
	\end{align*}
The increments $\btheta_{s,t}$ can be interpreted as an error of non-commutativity
between the increments of the shifted path $T_h\bX$ and the increments
of the (right-)translation $\bX\bh$ by the lifted path $\bh$. 
\end{remark}

We now use the shift map to define a cost function on $\bOmega_\bbG$.

\begin{definition}\label{def:rp-cost}
Define a cost function $C_\cH \colon \bOmega_{\mathbb{G}}\times\bOmega_{\mathbb{G}}\to [0,\infty]$ by
\begin{equation}\label{eq:rp-cost}
	C_\cH(\bX, \bY) \coloneqq
	\begin{cases}
		\|h\|_\cH, & \text{if} \; \bY = T_h \bX, \; \text{for some} \; h \in \cH,\\
        + \infty, & \text{otherwise}.
    \end{cases}
\end{equation}
\end{definition}

\begin{lemma}\label{lem:cost-mb}
    The cost $C_{\cH}\colon\bOmega_{\mathbb{G}}\times\bOmega_{\mathbb{G}}\to [0,\infty]$ is lower semicontinuous.
\end{lemma}

\begin{proof}
    Let $\bX, \bY \in \bOmega_\bbG$ and let $(\bX^n), (\bY^n) \subset \bOmega_\bbG$ be sequences such that $(\bX^n, \bY^n) \to (\bX, \bY)$. We may assume that there exists a subsequence $n_k\to\infty$ such that $\bY^{n_k} = T_{h^{n_k}} \bX^{n_k}$, where $h^{n_k} = \pi_1 \bY^{n_k} - \pi_1 \bX^{n_k} \in \cH$, and $\liminf_{n\to\infty}C_\cH(\bX^{n}, \bY^{n}) = \lim_{k\to\infty} C_\cH(\bX^{n_k}, \bY^{n_k})=:I<\infty$.
    Then we have that
    \begin{align*}
        I = \lim_{k \to \infty} C_\cH(\bX^{n_k}, \bY^{n_k}) = \liminf_{k \to \infty} \|h^{n_k}\|_\cH \geq \|h\|_\cH,
    \end{align*}
    where $h = \pi_1 \bY - \pi_1 \bX \in \cH$. By the continuity of the shift shown in \Cref{prop:shift} (iv), $\bY = T_h \bX$, and so $C_\cH(\bX, \bY) = \|h\|_\cH$.
\end{proof}

\subsection{Discussion on the choice of the cost function \texorpdfstring{$C_{\cH}$}{C\_H}}
\label{sec:sharp-cost}
The choice of cost function $C_\cH$ is natural in the sense that
it arises as the $\Gamma$-limit
of the sequence $C_n$, as shown in \Cref{sec:gamma-conv}. Moreover, $C_\cH$ has the crucial property that whenever $H(\bnu \| \bmu) = + \infty$, also $\T_{C_\cH, 2}(\bmu, \bnu) = +\infty$. Indeed, supposing that $\T_{C_\cH, 2}(\bmu, \bnu) < \infty$, there exists a coupling $\blambda \in \Pi(\bmu, \bnu)$ such that
    \begin{align*}
    	\blambda(\{\, (\bomega, \ol \bomega) \in \bOmega_\bbG \times \bOmega_\bbG : \; \ol \bomega = T_{\ol \omega - \omega} \bomega, \, \ol \omega - \omega \in \cH \, \}) = 1,
    \end{align*}
    and $\int C_\cH^2(\bomega, \ol \bomega)\di \blambda(\bomega, \ol \bomega) < \infty$.
    We have that $\bmu = \Psi_\sharp \mu$, by \cref{prop:BM-lift}, and combining this with \cref{eq:lift-shift} from \cref{prop:shift} gives
    \begin{align*}
    	\blambda(\{\, (\bomega, \ol \bomega) \in \bOmega_\bbG \times \bOmega_\bbG : \; \bomega = \Psi(\omega), \, \ol \bomega = \Psi(\ol \omega), \, \ol \omega - \omega \in \cH \, \}) = 1.
    \end{align*}
    Thus there exists $\nu \in \Pc(\Omega)$ such that $\bnu = \Psi_\sharp \nu$. By It\^o representation and Girsanov's theorem, we also have that $\nu \ll \mu$. Hence $\bnu = \Psi_\sharp \nu \ll \Psi_\sharp \mu = \bmu$, and $H(\bnu \| \bmu) < \infty$.
    
    This is in contrast to the cost function $\wt C_\cH \coloneqq c_\cH \circ (\pi_1 \times \pi_1) \colon \bOmega_\bbG \times \bOmega_\bbG \to [0, \infty]$, which appears in the cost-information inequality in \cite[Corollary 1.4]{riedel_transportationcost_2017}. Indeed, consider the Brownian motion $\bB = (B, \mathbb A)$ on the Heisenberg group $\bbG \cong \R^2 \oplus \R$. Let $\bnu = \Law(\bX)$, where $\bX = (B, \mathbb X)$ is defined as follows. Let $M \in (0, \infty)$, and define $\mathbb X_0 = 0$ and $\mathbb X_{s,t} = \mathbb A_{s,t} + (t - s)M$, for all $s, t \in [0, 1]$ with $s < t$. Since $\bX$ only differs from $\bB$ in the second component, we see that
    \begin{align*}
    	0 \leq \T_{\wt C_\cH, 2}^2(\bmu, \bnu) \leq  \E \Big[ \wt C_\cH^2(\bB, \bX) \Big] = \E \Big[ c_\cH^2(B, B) \Big] = 0.
    \end{align*}
    However, $\bnu$ is not absolutely continuous with respect to $\bmu$. Indeed, supposing that $\bnu \ll \bmu$, \Cref{lem:abs-cont-lift} implies that $\bnu = \Psi_{\sharp}\nu$ for some probability measure $\nu \ll \mu$. It follows that $\bnu(\Psi(\Omega)) = \nu(\Omega) = 1$. On the other hand, since $\bnu = \Law(\bX)$ with $\bX = (B, \mathbb X) \neq \Psi(B)$, we see that $\bnu(\Psi(\Omega)) < 1$, which is a contradiction. Hence $H(\bnu \| \bmu) = + \infty$.

\section{Talagrand for Brownian motion on Carnot groups -- Direct approach via F\"ollmer drift}\label{sec:top-foellmer}

In this section, we give a first proof of Talagrand's $\Tc_2$ inequality for the law of Brownian motion on a step-$2$ Carnot group. We follow the strategy of \cite{Le13} and \cite{Foe21}, using F\"ollmer's intrinsic drift from \cite{Fo86,Fo88}. Moreover, we show that equality is attained in the $\Tc_2$ inequality when restricting to adapted couplings, as was shown in the classical case in \cite[Lemma 5]{La18} and \cite[Theorem 3]{Foe21}.

We first give a characterisation of adapted couplings on $\Pc(\bOmega_\bbG \times \bOmega_\bbG)$. In particular, we show that our definition of adapted couplings in \Cref{def:adapted-coupling} is consistent with that of \cite[Definition 1]{Foe21}.

\begin{lemma}\label{lem:adapted-couplings}
    Suppose that $\nu \ll \mu$. Then $\lambda \in \Pi(\mu, \nu)$ is an adapted coupling if and only if there exists a filtered probability space $(\wt\Omega, \F, (\F_t), \Q)$ on which processes $X$, $Y$ are defined such that $X$ is a Brownian motion, $Y$ is an adapted process, and $\lambda = \Law(X, Y)$ under $\Q$.

    Suppose that $\bnu \ll \bmu$. Then $\blambda \in \Pi(\bmu, \bnu)$ is an adapted coupling if and only if there exists $\nu \ll \mu$ and an adapted coupling $\lambda \in \Pi_\ad(\mu, \nu)$ such that $\blambda = (\Psi \times \Psi)_\sharp \lambda$.
\end{lemma}

\begin{proof}
	Suppose that $\nu \ll \mu$. By \cite[Lemma 4]{La18}, our \Cref{def:adapted-coupling} of adapted couplings is equivalent to the symmetric counterpart of \cite[Definition 1]{La18} (see \cite[Section 4.1]{La18}). Then the first claim follows from \cite[Propositions 3 and 4]{La18}.
    
    Now suppose that $\bnu \ll \bmu$. By \Cref{lem:abs-cont-lift}, there exists $\nu \ll \mu$ such that $\bnu = \Psi_\sharp \nu$. If $X$ is an $\R^{d_1}$-valued process with natural filtration $(\F_t)_{t \in [0, 1]}$ completed with respect to the law of $X$, and $\bX = \Psi(X)$ is a $\bbG$-valued process with natural filtration $(\bfilt_t)_{t \in [0, 1]}$ completed with respect to the law of $\bX$, then $\bfilt_t = \Psi(\F_t)$, for all $t \in [0, 1]$. Since $\bmu = \Psi_\sharp \mu$ and $\bnu = \Psi_\sharp \nu$, we have that $\blambda \in \Pi_\ad(\bmu, \bnu)$ if and only if $\blambda = (\Psi \times \Psi)_\sharp \lambda$ for some $\lambda \in \Pi_\ad(\mu, \nu)$.
\end{proof}

\begin{remark}
    By \Cref{lem:adapted-couplings}, if $\blambda \in \Pi_\ad(\bmu, \bnu)$, then there exist $\bX = \Psi(X)$, $\bY = \Psi(Y)$ defined on some filtered probability space $(\wt\Omega, \F, (\F_t), \Q)$ such that $\blambda = \Law_\Q(\bX, \bY)$, where $\bX$ is a Brownian motion on $\bbG$ and $\bY$ is an adapted $\bbG$-valued process, and $\Law_\Q$ denotes the law under $\Q$. Then, letting $\E_\Q$ denote expectation with respect to $\Q$, we have
    \begin{align*}
        \int C_\cH^2(\bomega, \ol \bomega) \di \blambda = \E_\Q[C_\cH^2(\bX, \bY)].
    \end{align*}
\end{remark}

We now prove the main result of this section, showing that $\bmu$ satisfies a $\Tc_2$ inequality, and that equality holds when restricting to adapted couplings.

\begin{theorem}\label{thm:t2-top-foellmer}
   Let $\bnu \ll \bmu$. 
    Then there exists $\nu \ll \mu$ such that $\bnu=\Psi_{\sharp}\nu$, and there exists a predictable process $b^\nu$ on $\R^d$ with $b^\nu \in L^2$, $\nu$-almost surely, such that $B^\nu \coloneqq B - \int_0^\cdot b^\nu_t \di t$ is a Brownian motion under $\nu$, and $\blambda^\ast = \Law_\nu(\Psi(B^\nu),\Psi(B))$ is the unique optimal adapted coupling of $\bmu$ and $\bnu$ with
    \begin{align}\label{eq:claim1}
        \T_{C_{\mathcal{H}}, 2}(\bnu,\bmu)^2 \leq \T_{C_{\mathcal{H}}, 2}^{\ad}(\bmu,\bnu)^2 = \E_{\nu}[C_{\mathcal{H}}^2(\Psi(B^\nu),\Psi(B))] = 2H(\bnu \| \bmu).
    \end{align}
    In particular, $\bmu \in \Tc_2(\bOmega_\bbG, C_\cH, 1)$.
\end{theorem}

\begin{proof}
	By \Cref{lem:abs-cont-lift}, there exists a probability measure $\nu$ on $\Omega$ with $\nu \ll \mu$ and $\Psi_\sharp \nu = \bnu$. Writing $\brho\coloneqq\frac{\D\bnu}{\D\bmu}$, we have $\frac{\D\nu}{\D\mu}=\brho \circ \Psi$. Since $\bmu(\bOmega_\bbG \setminus \Psi(\Omega)) = 0$, by \Cref{prop:BM-lift}, we obtain that
    \begin{align*}
        H(\bnu \| \bmu)=\int_{\Psi(\Omega)} \brho (\bomega) \log (\brho(\bomega))\di\bmu(\bomega)=\int_\Omega \brho (\Psi(\omega)) \log (\brho(\Psi(\omega))\di\mu(\omega)=H(\nu \| \mu).
    \end{align*}
    We can apply \cite[Proposition 1]{Foe21} to obtain that there exists a predictable process $b^\nu$ with $b^\nu \in L^2$, $\nu$-almost surely, such that $B^\nu = B - \int_0^\cdot b^\nu_t \di t$ is a Brownian motion under $\nu$ with
    \begin{align}\label{eq:rel-entropy-id}
        \E_{\nu}[\|B^\nu - B\|_{\mathcal{H}}^2] = 2 H(\nu \| \mu) = 2 H(\bnu \| \bmu).
    \end{align}
    
    Moreover, from \cite[Theorem 3]{Foe21} it follows that $\lambda^\ast = \Law_\nu(B^\nu, B)$ is the unique optimal adapted coupling between $\mu$ and $\nu$. By \cref{lem:adapted-couplings}, $\blambda^\ast = \Law_\nu(\Psi(B^\nu),\Psi(B))$ is an adapted coupling of $\bmu=\Psi_{\sharp}\mu$ and $\bnu=\Psi_{\sharp}\nu$. Thus, using the definition of $C_{\mathcal{H}}$ from \cref{eq:rp-cost},
    \begin{align}\label{eq:opt-rp}
        \T_{c_{\mathcal{H}}, 2}^{\ad}(\mu,\nu)^2 = \E_{\nu}[\|B^\nu - B\|_{\mathcal{H}}^2]=\E_{\nu}[C_{\mathcal{H}}^2(\Psi(B^\nu),\Psi(B))] \geq \T_{C_\cH, 2}^\ad(\bmu, \bnu)^2.
    \end{align}

    On the other hand, since $c_\cH \circ (\pi_1 \times \pi_1) \leq C_\cH$, applying \Cref{lem:adapted-couplings} gives
    \begin{align*}
    	\T_{c_\cH, 2}^\ad(\mu, \nu)^2 = \inf_{\lambda \in \Pi_\ad(\mu, \nu)}\int_{\Omega \times \Omega} c_\cH^2(\omega, \ol \omega) \di \lambda(\omega, \ol \omega) & = \inf_{\blambda \in \Pi_\ad(\bmu, \bnu)}\int_{\bOmega_\bbG \times \bOmega_\bbG} c_\cH^2(\pi_1\bomega, \pi_1\ol \bomega) \di \blambda(\bomega, \ol \bomega)\\
    	& \leq \inf_{\blambda \in \Pi_\ad(\bmu, \bnu)}\int_{\bOmega_\bbG \times \bOmega_\bbG} C_\cH^2(\bomega, \ol \bomega) \di \blambda(\bomega, \ol \bomega)\\
        & = \T_{C_\cH, 2}^\ad(\bmu, \bnu)^2.
    \end{align*}
    Hence we have optimality of $\blambda^\ast$ and
    \begin{align*}
    	\T^{\ad}_{C_{\mathcal{H}}, 2}(\bnu,\bmu)^2 = \E_{\nu}[C_{\mathcal{H}^2}(\Psi(B^\nu),\Psi(B))] = \E_{\nu}[\|B^\nu - B\|_{\mathcal{H}}^2] = \T_{c_{\mathcal{H}}, 2}^{\ad}(\mu,\nu)^2.
    \end{align*}
    Applying \eqref{eq:rel-entropy-id} gives \eqref{eq:claim1}, and uniqueness of the optimiser follows by \Cref{lem:adapted-couplings}.
\end{proof}

\begin{remark}\label{rem:drift-adapted-shift}
	Suppose that $h$ is an adapted process with $h \in \cH$ almost surely, and let $\bnu = \Law(T_h\bB)$. Then we can take $b^\nu = \dot h$ in \Cref{thm:t2-top-foellmer}. Indeed, by Girsanov's theorem, $B - h$ is a Brownian motion under $\nu$, and  $H(\bnu \| \bmu) = H(\nu \| \mu) = \E_\nu[\log \D \nu/\D \mu] = \E_\nu[\|h\|_\cH^2]$.
\end{remark}

\section{Talagrand for Gaussian rough paths -- Direct approach via contraction}\label{sec:top-riedel}
We now give an alternative proof of Talagrand's $\Tc_2$ inequality in a more general setting, following the contraction approach of \cite{riedel_transportationcost_2017}.

Let $Z$ be a $d$-dimensional continuous Gaussian process that admits a level $2$ ``rough path'' lift $\boldsymbol{Z} = \boldsymbol{Z} (\omega)$ with $\boldsymbol{Z} \in \mathcal{D}$ almost surely,  where $\mathcal{D}$ is a suitable $p$-variation (or $\beta$-Hölder) rough path space
(cf.\ \cite[Chapter 15]{FrVi10}). 
We refer to \cite[Chapter 15]{FrVi10} for conditions under which a Gaussian process can be lifted to a Gaussian rough path. Here we simply assume that such a lift exists. Let $\nu \in \Pc(\Omega)$ denote the law of the Gaussian process $Z$, $\bnu \in \Pc(\mathcal D)$ the law of $\boldsymbol{Z}$, and $\cH_\nu$ the Cameron--Martin space of $Z$ (cf.\ \cite[Chapter 8, Section 4]{Janson97}).

A $\Tc_2$ inequality is known to hold for general Gaussian processes with $\alpha=1$ and cost 
\begin{align*}
    \wt{c}_{\cH_{\nu}} (x,y)=\begin{cases}
		\|h\|_{\cH_{\nu}}, & \text{if} \; x-y\in\mathcal{H}_{\nu},\\
        + \infty, & \text{otherwise};
    \end{cases}
\end{align*}
that is $\nu\in \Tc_2(\Omega,\wt{c}_{\cH_{\nu}},1)$; see \cite[Theorem 3.1]{FeUe04} and \cite[Theorem 1.2]{riedel_transportationcost_2017}.

We work under the following assumption.

\begin{assumption}\label{ass}Suppose that there exists $\wt \Omega \subseteq \Omega$ with $\nu(\wt\Omega) = 1$ such that
	\begin{enumerate}[label = (\roman*)]
		\item There exists a Borel-measurable lift map $\wt\Psi \colon \Omega \to \mathcal{D}$ with $\pi_1\wt\Psi(x)=x$, for $x \in \wt \Omega$, where $\pi_1$ is the projection onto the first component, such that $\wt\Psi(Z)=\boldsymbol{Z}$ almost surely;
		\item There exists a continuous shift map
			\begin{align*}
				\mathcal{H}_{\nu}\times \mathcal{D}\to \mathcal{D}, \quad (h,\boldsymbol{x})\mapsto \wt{T}_{h}\boldsymbol{x},
			\end{align*}
			such that
			\begin{align}\label{eq:lift-shift-property}
				\wt{T}_h \wt\Psi(x)=\wt\Psi(x+h), \quad x\in \wt \Omega,\quad h\in\mathcal{H}_{\nu}.
			\end{align}
	\end{enumerate}
\end{assumption}

\begin{remark}\label{rem:T2-Gaussian-Euclidean}
By standard results \cite[Chapter 15]{FrVi10}, we see that 
\Cref{ass} is satisfied for $Z=B$  a fractional Brownian motion and its lift $\bB$ in the step-$2$ Carnot group $\bbG$, with path space $\mathcal{D}=C^{p-\var}_0([0,1],\mathbb{G})$, for $H \in (1/3,1/2]$, $p \in (1/H,3)$. Extensions to $H>1/4$ are possible, at the price of lifting $Z$ to a step-$3$ Carnot group; we do not give details for the sake of brevity.
\end{remark}

\begin{definition}
    Define the cost $\wt{C}_{\cH_{\nu}}\colon \mathcal{D}\times \mathcal{D}\to [0,\infty]$, similarly to \eqref{eq:rp-cost}, by
    \begin{equation}\label{eq:rp-cost2}
	\wt{C}_{\cH_{\nu}}(\bX, \bY) \coloneqq
	\begin{cases}
		\|h\|_{\cH_{\nu}}, & \text{if} \; \bY = \wt{T}_h \bX, \; \text{for some} \; h \in \cH_{\nu},\\
        + \infty, & \text{otherwise},
    \end{cases}
\end{equation}
where $\wt{T}_{h}$ is the shift from \cref{ass}.
\end{definition}
Due to the assumed continuity of the shift, we recover measurability of $\wt{C}_{\cH_{\nu}}$ (cf.\ \cref{lem:cost-mb}).
By applying the contraction principle from \cref{lem:contraction}, similarly to \cite{riedel_transportationcost_2017}, but for a different cost, we lift the Talagrand inequality to the rough path space $\mathcal{D}$.
\begin{theorem}\label{thm:top-t2-riedel-approach}
    Let \cref{ass} hold. Then $\bnu\in\Tc_2(\mathcal{D},\wt{C}_{\cH_{\nu}}, 1)$.
\end{theorem}
\begin{proof}
    The property \cref{eq:lift-shift-property} implies that for $h\in\cH_{\nu}$, a path $x\in \wt \Omega$ satisfies $x=y+h$ if and only if $\wt\Psi(x)=\wt{T}_{h}\wt\Psi(y)$. Hence for $x-y\in \mathcal{H}_{\nu}$, we have that $\wt\Psi(y+(x-y))=\wt{T}_{x-y}\wt\Psi(y)$ and thus
\begin{align*}
\|x-y\|_{\mathcal{H}_{\nu}}^2=\wt{C}_{\mathcal{H}_{\nu}}^2(\wt\Psi(x),\wt\Psi(y)).
\end{align*} 
Similarly, if $x-y\notin\mathcal{H}_{\nu}$, then $\wt\Psi(x)$ is not a shift of $\wt\Psi(y)$ and so $\wt{C}_{\mathcal{H}_{\nu}}(\wt\Psi(x),\wt\Psi(y))= +\infty$. Together with $\nu\in\Tc_2(\Omega, \wt{c}_{\mathcal{H}_{\nu}},1)$, $\nu(\wt \Omega)= 1$, and measurability of $\wt{C}_{\cH_{\nu}}$ and $\wt\Psi$, an application of the contraction principle (\cref{lem:contraction}) with $L=1$ then yields $\bnu\in\Tc_2(\mathcal{D},\wt{C}_{\mathcal{H}_{\nu}},1)$. 
\end{proof}

\section{Talagrand for the heat kernel measure on Carnot groups}\label{sec:heat-kernel}

In this section, we show that a $\Tc_2$ inequality on the step-$2$ Carnot group $\bbG$ follows from a log-Sobolev inequality, which in turn can be deduced from a heat semigroup estimate. In particular, we prove a $\Tc_2$ inequality for the heat kernel measure $\bmu_1$ in the case that $\bbG$ is an H-type group. We will apply this result in \cref{sec:bottom-up} to show that a $\Tc_2$ inequality also holds on the path space by a bottom-up approach.

\subsection{From log-Sobolev to Talagrand}\label{sec:log-sob}
We follow the approach put forward by Otto--Villani \cite{OtVi00}, namely deducing a $\Tc_2$ inequality as a consequence of a log-Sobolev inequality. We will make use of the generalisation by Gigli--Ledoux \cite{GiLe13} of Otto--Villani's result. Whereas the result of Gigli--Ledoux \cite{GiLe13} depends on the log-Sobolev inequality for Lipschitz test functions, we show via a mollification argument that this can be relaxed to only requiring the log-Sobolev inequality for smooth test functions; see \cref{thm:gigli-ledoux}.  Moreover, for the heat kernel measure, we show that the log-Sobolev inequality for smooth test functions follows from certain heat semigroup estimates; see \cref{thm:heat-estimate-t2}.
In the special case of H-type groups, as defined in \Cref{rem:HtypeGroups}, the required heat semigroup estimates are known. Thus, in \cref{thm:t2-h-type}, we show that a $\Tc_2$ inequality holds for the heat kernel measure on H-type groups and, in particular, on the Heisenberg group.

Let $\bbG$ be a step-$2$ Carnot group. Recall that, for the Carnot--Carath\'eodory metric $d_\cc$ defined in \eqref{eq:dcc}, the space $(\bbG, d_\cc)$ is a Polish space. Hence, for any Borel probability measure $\etab$ on $\bbG$, the space $(\bbG, d_\cc, \mathcal \etab)$ is a metric measure space in the sense of \cite{GiLe13}.

For a locally Lipschitz function $f \colon \bbG \to \R$, define the local Lipschitz constant $\mathrm{Lip}_\bbG(f)$ by
\begin{equation*}
    \mathrm{Lip}_\bbG(f)(\bx) \coloneqq \limsup_{\by \to \bx}\frac{|f(\by) - f(\bx)|}{d_\cc(\bx, \by)}, \quad \bx \in \bbG.
\end{equation*}
By an extension of Rademacher's theorem due to Pansu \cite{P1989MCQE} (see also \cite{de_philippis_converse_2025,pinamonti_porosity_2017}, and \cite[Theorem 11.3.2]{LeDonne24}),
we obtain that every Lipschitz continuous function $f \colon U\subseteq\bbG\to\R$ is Pansu differentiable Lebesgue-almost everywhere. 
In particular, its gradient $\nabla_\bbG f$ exists $\mathcal L^m$-almost everywhere and $\mathrm{Lip}_\bbG(f)(\bx) = |\nabla_\bbG f(\bx)|_\bbG$, for $\mathcal L^m$-almost every $\bx \in \bbG$.

We say that a Borel probability measure $\etab$ on $\bbG$ with $\etab\ll \mathcal L^m$ satisfies the log-Sobolev inequality if there exists $\alpha \in (0, \infty)$ such that
    \begin{equation}\label{eq:log-sob}\tag{LSI}
       2 \alpha \int_{\bbG} f \log f \di \etab \leq \int_{\{f > 0\}} \frac{|\nabla_\bbG f|_\bbG^2}{f}\di \etab, \quad \text{for all} \; f \in C^\infty_c(\bbG, [0, \infty)), \int_\bbG f \di \etab = 1.
    \end{equation}
    
We first show that \eqref{eq:log-sob} implies a $\Tc_2$ inequality. As an intermediate step, we show that the log-Sobolev inequality also holds for Lipschitz functions, so that we can then apply \cite[Theorem 5.2]{GiLe13}.

\begin{theorem}\label{thm:gigli-ledoux}
    Suppose that $\etab$ is a Borel probability measure on $\bbG$ with $\etab\ll \mathcal L^m$ satisfying \cref{eq:log-sob} for some $\alpha \in (0, \infty)$. Then, for any Lipschitz function $f \colon \bbG \to [0, \infty)$ with $\int_{\bbG}f \di \etab = 1$,
    \begin{equation}\label{eq:log-sob-lipschitz}
        2 \alpha \int_{\bbG} f \log f \di \etab \leq \int_{\{f > 0\}} \frac{|\nabla_\bbG f|_\bbG^2}{f}\di \etab.
    \end{equation}
    Moreover, $\etab \in \Tc_2(\bbG, d_\cc, \alpha)$.
\end{theorem}

\begin{proof}
	We argue as in the proof of \cite[Theorem 4.8]{AS2019HEFC} via mollification. Consider $f\colon \bbG\to[0,\infty)$ Lipschitz continuous with $\int_\bbG f \di \etab = 1$. In particular, its gradient $\nabla_\bbG f$ exists $\mathcal L^m$-almost everywhere.
	
	Let $\rho \colon \bbG \to \R$ be a symmetric smooth mollifier in  $\bbG$, i.e.\ a function $\rho \in C^\infty_c(\R^m, [0, \infty))$ such that
    $\mathrm{supp} \rho \subset B_1$, $0 \leq \rho \leq 1$, $\rho(\bx^{-1}) = \rho(\bx)$, for all $\bx \in \bbG$,
    and $\int_{\bbG}\rho \di \bx = 1$.
    For $k\in\N$ and $\bx\in\bbG$, set $\rho_k(\bx) = k^Q \rho(\delta_k\bx)$ and define the mollification $\wh f_k = \rho_k \star f$ by
    \[
    	\wh f_k(\bx)\coloneqq (\rho_k \star f)(\bx) = \int_\bbG \rho_k(\bx\by^{-1}) f(\by)\di \by= \int_{\bbG} \rho_{k}(\by)f(\by^{-1}\bx)\di\by,\quad\bx\in\bbG.
    \]
    For each $k \in \N$, also define a smooth truncation function $\chi_k \in C^\infty_c(\R^m, [0,1])$ such that $\chi_k = 1$ in $B_k$, $\chi_k = 0$ in $B_{2k}^c$, and $|\nabla_\bbG \chi_k|_\bbG \leq C/k$ for some constant $C \in (0, \infty)$, and define $f_k \coloneqq \wh f_k \, \chi_k$.
    
    Thus, $f_k \in C^\infty_c(\bbG, [0, \infty))$, for each $k \in \N$. Moreover, $f_k \to f$ in $L^1(\etab)$ and thus $\etab$-almost everywhere along a subsequence. By Fatou's lemma, we have that
    \begin{align*}
    	\liminf_{k \to \infty}\int_\bbG f_k \log f_k \di \etab \ge \int_\bbG f \log f \di \etab.
    \end{align*}
    
    We next show an upper estimate for the $\limsup$ of the right-hand side in \cref{eq:log-sob-lipschitz} with $\wh f_k$ in place of $f$. Note that, by left-invariance of the Carnot--Caratheodory distance, for $\bx_1, \bx_2 \in \bbG$, we have
    \[
    \frac{|\wh f_k(\bx_1)-\wh f_k(\bx_2)|}{d_\cc(\bx_1,\bx_2)}\leq
    \int_\bbG \rho_k(\by) \frac{|f(\by^{-1}\bx_1)- f(\by^{-1}\bx_1)|}{d_\cc(\bx_1,\bx_2)}\di \by
    \leq   \int_\bbG \rho_k(\by) \frac{|f(\by^{-1}\bx_1)- f(\by^{-1}\bx_1)|}{d_\cc(\by^{-1}\bx_1,\by^{-1}\bx_2)}\di \by.
    \]
    Thus, after passing to the $\limsup$ for $\bx_1\to\bx_2=\bx$, 
    we obtain $\mathrm{Lip}_\bbG \wh f_k(\bx) \leq \rho_k \star \mathrm{Lip}_\bbG f(\bx)$,
    or equivalently $|\nabla_\bbG \wh f_k|_\bbG(\bx) \leq \rho_k\star|\nabla_\bbG f|_\bbG (\bx)$.
    With this estimate, the Cauchy--Schwarz inequality gives
    \[
    |\nabla_\bbG \wh f_k|_\bbG^2\leq \Big[ \rho_k\star\Big(\chi_{\{f>0\}}\sqrt{f}\frac{|\nabla_\bbG f|_\bbG}{\sqrt{f}}\Big)\Big]^2
    \leq f_k \Big(\rho_k\star \frac{|\nabla_\bbG f|_\bbG^2}{f}\chi_{\{f>0\}}\Big).
    \]
    Multiplying $f_k$ by the truncation $\chi_k$ and applying the product rule, we estimate
    \begin{align*}
    	\limsup_{k \to \infty}\int_{\{f_k > 0\}} \frac{|\nabla_\bbG f_k|_\bbG^2}{f_k}\di \etab & \leq \limsup_{k \to \infty}\int_{\{\wh f_k > 0\}} \frac{|\nabla_\bbG \wh f_k|_\bbG^2}{\wh f_k}\di \etab\\
    	& \leq \limsup_{k \to \infty} \int_{\bbG} \rho_k\star\bigg(\frac{|\nabla_\bbG f|_\bbG^2}{f}\chi_{\{f > 0\}}\bigg)\di \etab \leq \int_{\{f > 0\}} \frac{|\nabla_\bbG f|_\bbG^2}{f}\di \etab.
    \end{align*}
       	    
    Finally, define $I_k \coloneqq \int_\bbG f_k \di \etab$ and the normalised function $\wt f_k \coloneqq f_k / I_k $, for each $k \in \N$. Then, by \eqref{eq:log-sob}, 
    \begin{align*}
    	2 \alpha \int_\bbG f_k \log f_k \di \etab & = 2 \alpha I_k \int_\bbG \wt f_k \log \wt f_k \di \etab + 2 \alpha I_k \log I_k\\
    	& \leq I_k \int_{\{\wt f_k > 0\}}\frac{|\nabla_\bbG \wt f_k|_\bbG^2}{\wt f_k} \di \etab + 2 \alpha I_k \log I_k
    	= \int_{\{f_k > 0\}} \frac{|\nabla_\bbG f_k|_\bbG^2}{f_k} \di \etab + 2 \alpha I_k \log I_k.
    \end{align*}
    We conclude that
    \[
    	2 \alpha  \int_\bbG f \log f \di \etab \le 
        \liminf_{k \to \infty} 2 \alpha \int_\bbG f_k \log f_k \di \etab
        \leq \limsup_{k\to\infty}
        \int_{\{f_k > 0\}} \frac{|\nabla_\bbG f_k|_\bbG^2}{f_k}\di \etab
        \leq \int_{\{f > 0\}} \frac{|\nabla_\bbG f|_\bbG^2}{f}\di \etab.
    \]
    Thus, the log-Sobolev inequality for Lipschitz-continuous functions is established.
    
    Finally, since $(\bbG, d_\cc, \etab)$ is a metric measure space in the sense of \cite{GiLe13}, we conclude that $\etab \in \Tc_2(\bbG, d_\cc, \alpha)$ by \cite[Theorem 5.2]{GiLe13}.
\end{proof}

\subsection{From heat semigroup estimates to Talagrand}
We now specialise to the case of the heat kernel measure $\bmu_1$ and give a sufficient condition for $\bmu_1 \in \Tc_2(\bbG, d_\cc, \alpha)$, for some $\alpha \in (0, \infty)$.

As noted in \cite[Section 5]{El10} and \cite[Remark 6.6]{BaBaBoCh08}, one can deduce the log-Sobolev inequality \eqref{eq:log-sob} for $\bmu_1$ from the following heat semigroup estimate: there exists a constant $K \in (0, \infty)$ such that
\begin{align}\label{eq:heat-semigroup-estimate}
	|\nabla_\bbG P_t f|_\bbG \leq K P_t(|\nabla_\bbG f|_\bbG), \quad \text{for all} \; f \in C^\infty_c(\bbG, \R), \ t \geq 0,
\end{align}
where $P_t$ is the heat semigroup defined in \Cref{eq:semigroup}.
Indeed, for the Heisenberg group $\bbG = \bbH = \bbH^1$, \cite[Th\'eor\`eme 1.1]{Li06} proves the estimate \eqref{eq:heat-semigroup-estimate}, and \cite[Corollaire 1.2]{Li06} states that \eqref{eq:log-sob} holds as a direct consequence, following the arguments in \cite[Th\'eor\`eme 5.4.7]{logsob-fr-00}. More generally, for $\bbG = \bbH^n$, \cite[Theorem 7.3]{HebZeg2010CIMM} and \cite[Theorem 6.1]{BaBaBoCh08} prove that \eqref{eq:log-sob} holds, again relying on the heat semigroup estimate \eqref{eq:heat-semigroup-estimate}.
For completeness, we provide a proof in \Cref{thm:heat-estimate-t2} that, for any step-$2$ Carnot group, the heat semigroup estimate \eqref{eq:heat-semigroup-estimate} implies the log-Sobolev inequality \eqref{eq:log-sob}. Thanks to \Cref{thm:gigli-ledoux}, the $\Tc_2$ inequality also follows.

\begin{theorem}\label{thm:heat-estimate-t2}
	Let $\bbG$ be a step-$2$	 Carnot group and suppose that there exists $K \in (0, \infty)$ such that the heat semigroup $P$ on $\bbG$ satisfies the estimate \eqref{eq:heat-semigroup-estimate} for all $t \in [0, 1]$. Let $\alpha = \tfrac{1}{2K^2}$. Then the heat kernel measure $\bmu_1$ on $\bbG$ satisfies the log-Sobolev inequality \eqref{eq:log-sob} with constant $\alpha$, and $\bmu_1$ satisfies the $\Tc_2$ inequality $\bmu_1 \in \Tc_2(\bbG, d_\cc, \alpha)$.
\end{theorem}

\begin{proof}
	Suppose that $P_t$ satisfies \eqref{eq:heat-semigroup-estimate} for all $t \in [0, 1]$.
	Let $\phi \in C^2(I, \R)$ for some interval $I \subset \R$, and suppose moreover that $\phi^{\prime \prime} > 0$ and the function $-1/\phi^{\prime \prime}$ is convex. Let $f \in C^\infty_c(\bbG, \R)$ and let $t \in [0, 1]$.  Then, by the heat equation and chain rule for the sub-Laplacian, for any $s \in [0, t]$,
	\begin{align*}
		\partial_s P_s \phi(P_{t - s}f) & = P_s\bigl(\Delta_\bbG \phi(P_{t - s} f) - \phi^\prime(P_{t - s}f)\Delta_\bbG P_{t - s} f\bigr)\\
		& = P_s\bigl(\phi^{\prime \prime}(P_{t - s}f)|\nabla_\bbG P_{t - s} f|_\bbG^2\bigr).
	\end{align*}
	The heat semigroup estimate \eqref{eq:heat-semigroup-estimate} and the Cauchy--Schwarz inequality imply that
	\begin{align*}
		|\nabla_\bbG P_{t - s} f|_\bbG^2 \leq K^2 \big( P_{t - s}(|\nabla_\bbG f|_\bbG)\big)^2
		& = K^2 \Big( P_{t - s}\bigl(|\nabla_\bbG f|_\bbG\sqrt{\phi^{\prime \prime}(f)} \cdot 1/\sqrt{\phi^{\prime \prime}(f)}\bigr)\Big)^2\\
		& \leq K^2 P_{t - s}(|\nabla_\bbG f|_\bbG^2\phi^{\prime \prime}(f)) P_{t - s}(1/\phi^{\prime \prime}(f)),
	\end{align*}
	and, by Jensen's inequality,
	\begin{align*}
		\phi^{\prime \prime}(P_{t - s}f) & \leq \frac{-1}{P_{t - s}( - 1/\phi^{\prime \prime}(f))} = \frac{1}{P_{t - s}(1/\phi^{\prime \prime}(f))}.
	\end{align*}
	Hence
	\begin{align*}
		\partial_s P_s \phi(P_{t - s}f) & = P_s\bigl(\phi^{\prime \prime}(P_{t - s}f)|\nabla_\bbG P_{t - s} f|_\bbG^2\bigr)\\
		& \leq K^2 P_s P_{t - s}(|\nabla_\bbG f|_\bbG^2\phi^{\prime \prime}(f))
		 = K^2 P_t (|\nabla_\bbG f|_\bbG^2\phi^{\prime \prime}(f)),
	\end{align*}
	and integrating gives
	\begin{align*}
		P_t \phi(f) - \phi(P_t f) & = \int_0^t \partial_s P_s \phi(P_{t - s}f) \di s
		\leq K^2 t P_t (|\nabla_\bbG f|_\bbG^2\phi^{\prime \prime}(f)).
	\end{align*}
	
	Now take $\phi \colon (0, \infty) \to \R$ to be $\phi(x) = x \log x$ for all  $x \in (0, \infty)$, and suppose that $f \colon \bbG \to (0, \infty)$. Then we have the following form of the log-Sobolev inequality:
	\begin{align*}
		P_t (f \log f) - P_t f \log (P_t f) \leq K^2 t P_t \frac{|\nabla_\bbG f|_\bbG^2}{f}.
	\end{align*}
	To arrive at \eqref{eq:log-sob}, we set $t = 1$, evaluate both sides of the inequality at the identity, and additionally suppose that $\int_\bbG f \di \bmu_1 = 1$. Then
	\begin{align*}
		\int_\bbG f \log f \di \bmu_1 \leq K^2 \int_\bbG\frac{|\nabla_\bbG f|_\bbG^2}{f} \di \bmu_1.
	\end{align*}
	Note that, allowing $f \colon \bbG \to [0, \infty)$, we have $\int_\bbG f \log f \di \bmu_1 = \int_{\{f > 0\}} f \log f \di \bmu_1$. Thus \eqref{eq:log-sob} holds with constant $\alpha = \tfrac{1}{2 K^2}$.
	
	Applying \Cref{thm:gigli-ledoux}, we further have that $\bmu_1 \in \Tc_2(\bbG, d_\cc, \alpha)$.
\end{proof}

\subsection{Talagrand on Heisenberg-type groups}\label{sec:t2-h-type}
For any H-type group $\bbG$, \cite[Theorem 2.4]{El10} proves that the heat semigroup estimate \eqref{eq:heat-semigroup-estimate} is satisfied.
We thus have the following corollary of \cref{thm:heat-estimate-t2}.

\begin{theorem}\label{thm:t2-h-type}
    Let $\bbG$ be an H-type group. Then there exists $\alpha > 0$ such that the heat kernel measure $\bmu_1$ on $\bbG$ satisfies the log-Sobolev inequality \eqref{eq:log-sob} with constant $\alpha$, and $\bmu_1$ satisfies the $\Tc_2$ inequality $\bmu_1 \in \Tc_2(\bbH, d_\cc, \alpha)$.
\end{theorem}

\begin{proof}
    By \cite[Theorem 2.4]{El10}, the estimate \eqref{eq:heat-semigroup-estimate} holds on $\bbG$ with some constant $K$. Thus the result follows from \cref{thm:heat-estimate-t2} with $\alpha = \frac{1}{2 K^2}$.
\end{proof}

 We remark that the best possible constant in \Cref{thm:t2-h-type} is $\alpha\leq 1/2$, since \cite[Proposition 4.1]{El10} shows that the optimal constant in \Cref{eq:heat-semigroup-estimate} satisfies $K \ge \sqrt{\tfrac{3 d_1 + 5}{3 d_1 + 1}}$.

\section{Talagrand for Brownian motion on Carnot groups -- Bottom-up approach}\label{sec:bottom-up}

In Section \ref{sec:heat-kernel},
we discussed the availability of Talagrand transportation inequalities
on Carnot groups, as a consequence of log-Sobolev inequalities and heat kernel estimates. 
In this section, we demonstrate 
that we can transfer the $\Tc_2$ inequality for the heat kernel measure on a Carnot group, via a rescaling and tensorisation argument, to a $\Tc_2$ inequality on the associated path space; see \cref{ss:limitPassageTalagrand}. 
We highlight that this approach yields interesting insights into
optimal transport problems in the non-commutative sub-Riemannian setting that distinguishes it from the Euclidean case; see \cref{sec:projection-blow-up,sec:riemann}.
Finally, we show that the cost function defined in \eqref{eq:rp-cost} on the path space arises naturally 
as the $\Gamma$-limit of discretised cost functions
based on the Carnot--Caratheodory distance on the Carnot group; see \cref{sec:gamma-conv}.

Throughout this section, let $\bbG \cong \R^{d_1} \oplus \R^{d_2}$ be a step-$2$ Carnot group and set $d = d_1$; see \cref{sec:carnot-step-2}. 
Recall that $\bB$ denotes Brownian motion on $\bbG$, with law $\bmu = \Law(\bB)$ and time marginals $\bmu_t = \Law(\bB_t)$ for $t \in [0, 1]$.

\subsection{From Talagrand on Carnot groups to Talagrand on path space}
\label{ss:limitPassageTalagrand}
The main result of this section is that the $\Tc_2$ inequality for $\bmu_1$ on the group $\bbG$ implies the $\Tc_2$ inequality for $\bmu$ on the space $\bOmega_\bbG$ of continuous $\bbG$-valued paths started from the origin.

We first show that the heat kernel measure satisfies the following scaling property.

\begin{lemma}\label{lem:t2-dilation}
    Suppose that $\bmu_1 \in \Tc_2(\bbG, d_\cc, \alpha)$, for some $\alpha \in (0, \infty)$. Then, for any $t \in (0, 1]$, $\bmu_t \in \Tc_2(\bbG, d_\cc, \alpha t^{-1})$.
\end{lemma}

\begin{proof}
    We claim that the heat kernel measure on $\bbG$ satisfies the scaling
    \begin{equation}\label{eq:scaling}
        \bmu_{t} = (\delta_{s^{-1}})_{\sharp}\bmu_{s^2 t},
    \end{equation} for any $s > 0$, $t \in [0, 1]$. To see this, recall that $\mathfrak{p} \colon (0, \infty) \times \bbG \to (0, \infty)$ denotes the heat kernel on $\bbG$, $Q$ the homogeneous dimension of $\bbG$, and $\mathcal L^m$ the Lebesgue measure on $\bbG$. Then, as in \cref{eq:homogeneity}, $(\delta_s)_{\sharp}\mathcal L^m(\D \bx) = s^{-Q} \mathcal L^m(\D \bx)$,  for any $s >0$. Moreover, by \cite[Theorem 2.3]{AS2019HEFC}, for any $s>0$, $t \in [0, 1]$, and $\bx \in \bbG$, we have that $\mathfrak{p}_{s^2 t}(\delta_s \bx) = s^{-Q} \mathfrak{p}_t(\bx)$. Therefore, for any Borel set $A \subseteq \bbG$,
    \begin{align*}
        \mu_t (A) = \int_A \mathfrak{p}_t(\bx) \mathcal L^m(\D \bx)
        & = s^Q \int_A (\mathfrak{p}_{s^2 t} \circ \delta_s)(\bx) \mathcal L^m(\D \bx)\\
        & = s^Q \int_{\delta_{s^{-1}}(A)} \mathfrak{p}_{s^2 t}(\bx) (\delta_s)_{\sharp} \mathcal L^m (\D \bx)
        = \int_{\delta_{s^{-1}}(A)} \mathfrak{p}_{s^2 t}(\bx) \mathcal L^m(\D \bx) = (\delta_{s^{-1}})_{\sharp} \mu_{s^2 t}(A).
    \end{align*}
    This proves the claim. Now fix $t \in (0, 1]$. 
    Setting $s = t^{- \frac{1}{2}}$ in \cref{eq:scaling}, we have $\bmu_t = (\delta_{\sqrt t})_{\sharp}\bmu_1$.
    The map $\delta_{\sqrt{t}} \colon \bbG \to \bbG$ is $L$-Lipschitz with Lipschitz constant $L = \sqrt t$.
    Thus, since $\bmu_1 \in \Tc_2(\bbG, d_\cc, \alpha)$, \cref{lem:contraction} implies that $\bmu_t \in \Tc_2(\bbG, d_\cc, \alpha t^{-1})$.
\end{proof}

We now consider the product space $\bbG^{2^n}= \bbG\times \cdots\times \bbG$, for some $n \in \N$, and apply the dimension-free tensorisation property of the $\Tc_2$ inequality. Define $d_{\cc, n} \colon \bbG^{2^n} \times \bbG^{2^n} \to [0, \infty)$ by
\begin{equation*}
	d_{\cc, n}^2(\bx, \ol\bx) \coloneqq 2^n \sum_{i = 1}^{2^n} d_\cc^2(\bx_i, \ol \bx_i),
\end{equation*}
for $\bx = (\bx_1, \dotsc, \bx_n)$, $\ol \bx = (\ol \bx_1, \dotsc, \ol \bx_n) \in \bbG^{2^n}$.
We have the following tensorisation result.

\begin{proposition}\label{prop:t2-tensorisation}
     Suppose that $\bmu_1 \in \Tc_2(\bbG, d_\cc, \alpha)$, for some $\alpha \in (0, \infty)$. Then, for any $n \in \N$, $\bmu_{2^{-n}}^{\otimes 2^n} \in \Tc_2(\bbG^{2^n}, d_{\cc, n}, \alpha)$.
\end{proposition}

\begin{proof}
    Fix $n \in \N$. First note that $\bmu_{2^{-n}} \in \Tc_2(\bbG, d_\cc, \alpha 2^n)$, by \cref{lem:t2-dilation}. Define $\wt d_{\cc, n} \colon \bbG^{2^n} \to [0, \infty)$ by
    \begin{equation*}
        \wt{d}_{\cc, n}^2(\bx, \ol \bx) = \sum_{i = 1}^{2^n} d_\cc^2(\bx_i, \ol \bx_i),
    \end{equation*}
    for $\bx = (\bx_1, \dotsc, \bx_n)$, $\ol \bx = (\ol \bx_1, \dotsc, \ol \bx_n) \in \bbG^{2^n}$. Since $(\bbG, d_\cc)$ is a Polish space, \cite[Theorem 6]{GoLe07} implies that $\Tc_2(\bbG, d_\cc, \alpha 2^n)$ has the dimension-free tensorisation property; i.e.\ $\bmu_{2^{-n}}^{\otimes 2^n} \in \Tc_2(\bbG^{2^n}, \wt d_{\cc, n}, \alpha 2^n)$. Applying \cref{lem:contraction} with $\psi$ equal to the identity, we conclude that $\bmu_{2^{-n}}^{\otimes 2^n} \in \Tc_2(\bbG^{2^n}, d_{\cc, n}, \alpha)$.
\end{proof}

We next prove a relative entropy bound for measures on the path space.
For $n \in \N$, set $t_k^n = k 2^{-n}$ for $k \in \{0, \dotsc, 2^n\}$, and define $\Gamma^n \colon \bOmega_\bbG \to \bbG^{2^n}$ to be the projection of paths to their dyadic increments; i.e.~$\Gamma^n \bomega = (\bomega_{0, t_1^n}, \bomega_{t_1^n, t_2^n}, \dotsc, \bomega_{t_{2^n - 1}^n, 1})$, for $\bomega \in \bOmega_\bbG$. Then define a cost function $C_n \colon \bOmega_\bbG \times \bOmega_\bbG \to [0, \infty)$ by
\begin{equation}\label{eq:defCostDisc}
	C_n(\bomega, \ol \bomega) \coloneqq d_{\cc, n}(\Gamma^n\bomega,\Gamma^n\ol\bomega), \quad \bomega, \ol\bomega \in \bOmega_\bbG.
\end{equation}

\begin{lemma}\label{lem:rel-entropy-bound}
    Suppose that $\bmu_1 \in \Tc_2(\bbG, d_\cc, \alpha)$, for some $\alpha \in (0, \infty)$, and let $\bnu$ be a Borel probability measure on $\bOmega_\bbG$. Define $\bmu^n = \Gamma^n_\sharp \bmu$, $\bnu^n = \Gamma^n_\sharp \bnu \in \Pc(\bbG^{2^n})$.
    Then $H(\bnu^n \| \bmu^n) \nearrow H(\bnu \| \bmu)$ as $n \to \infty$ and, for any $n \in \N$,
    \begin{equation*}
         \T_{C_n, 2}(\bmu, \bnu) \leq \sqrt{\frac{2}{\alpha} H(\bnu \| \bmu)};
    \end{equation*}
    i.e.\ $\bmu \in \Tc_2(\bOmega_\bbG, C_n, \alpha)$.
\end{lemma}

\begin{proof}
    By independence and stationarity of the increments of $\bB$, we have that $\bmu^n = \Gamma^n_{\sharp} \bmu = \bmu_{2^{-n}}^{\otimes 2^n}$ and thus, by \cref{prop:t2-tensorisation}, $\bmu^n \in \Tc_2(\bbG^{2^n}, d_{\cc, n}, \alpha)$.
	
	As shown in \cite[Lemma 2.1]{DjGuWu04},
	\begin{equation}\label{eq:rel-ent-pusforward}
		H(\bnu^n \| \bmu^n) = \inf \Big\{\, H(\etab \| \bmu) : \etab \in \Pc(\bOmega_\bbG), \, \Gamma^n_{\sharp} \etab = \bnu^n \, \Big\},
	\end{equation}
	and we see that the right-hand side is increasing in $n$ and bounded above by $H(\bnu \| \bmu)$. To see that the limit is equal to $H(\bnu \| \bmu)$, we introduce the piecewise linear interpolation map $R^n \colon \bbG^{2^n} \to \bOmega_\bbG$, which is defined such that $\Gamma^n \circ R^n = \id$, and the image $\mathrm{Im}(R^n) \subset \bOmega_\bbG$ is the set of paths that are linear except at the dyadics $t^n_k$, $k \in \{0, \dotsc, 2^n\}$. Define $\wt \bmu^n = R^n_\sharp \bmu^n$, $\wt \bnu^n = R^n_\sharp \bnu^n \in \Pc(\bOmega_\bbG)$. Since $R^n \colon \bbG^{2^n} \to \mathrm{Im}(R^n)$ is a bijection, applying the representation given in \eqref{eq:rel-ent-pusforward} for both $R^n$ and its inverse gives the equality $H(\bnu^n \| \bmu^n) = H(\wt \bnu^n \| \wt \bmu^n)$. We conclude similarly to \cite[Corollary 9.4.6]{AGS}, as follows. For any $\bomega \in \bOmega_\bbG$, we have that $R^n \circ \Gamma^n (\bomega) \to \bomega$ as $n \to \infty$ and so, by dominated convergence, $\wt \bmu^n \rightharpoonup \bmu$ and $\wt \bnu^n \rightharpoonup \bnu$. Then, using the joint lower semicontinuity of the relative entropy (see, e.g.\ \cite[Lemma 9.4.3]{AGS}) together with the upper bound implied by \eqref{eq:rel-ent-pusforward}, we conclude that
	\begin{align*}
		\lim_{n \to \infty} H(\bnu^n \| \bmu^n) = \lim_{n \to \infty}H(\wt \bnu^n \| \wt \bmu^n) = H(\bnu \| \bmu).
	\end{align*}
    Finally, for any $n \in \N$, $\bmu^n \in \Tc_2(\bbG^{2^n}, d_{\cc, n}, \alpha)$ implies that
    \begin{equation*}
        \T_{C_n, 2}^2(\bmu, \bnu) = \T_{d_{\cc, n}, 2}^2(\bmu^n, \bnu^n) \leq \frac{2}{\alpha} H(\bnu^n \| \bmu^n) \leq \frac{2}{\alpha} H(\bnu \| \bmu).\qedhere
    \end{equation*}
\end{proof}

Before turning to the main result of this section, we prove an auxiliary lemma on the Euclidean cost on $\R^d$ and the associated Cameron--Martin cost $c_\cH$ defined in \cref{eq:cameron-martin}. For $n \in \N$, define $c_n \colon \Omega \times \Omega \to [0, \infty)$ by
\begin{equation}\label{eq:cost-euclidean}
	c_n^2(\omega, \ol \omega) = 2^{n} \sum_{k = 1}^{2^n}|\ol \omega_{t_{k - 1}^n, t_k^n} - \omega_{t_{k - 1}^n, t_k^n}|^2,
\end{equation}
for $\omega, \ol \omega \in \Omega$. Part (ii) of the following lemma is a standard stability result from optimal transport and is a consequence, for example, of \cite[Lemma 1.1]{riedel_transportationcost_2017}. Part (iii) will also be used in \cref{prop:liminf}.

\begin{lemma}\label{lem:euclidean-liminf}~
    \begin{enumerate}[label=(\roman*)]
        \item For each $(\omega, \ol \omega) \in \Omega \times \Omega$, $(c_n(\omega, \ol \omega))_{n \in \N}$ is an increasing sequence, and 
			\begin{equation}\label{eq:mono-conv-euclidean}
				\lim_{n \to \infty}c_n(\omega, \ol \omega) = c_\cH (\omega, \ol \omega).
			\end{equation}		
		\item For any $\nu \in \Pc(\Omega)$, the following convergence holds along a subsequence:
			\begin{equation}\label{eq:stability-euclidean}
				\lim_{n \to \infty}\T_{c_n, 2}(\mu, \nu) = \T_{c_\cH, 2}(\mu, \nu).
			\end{equation}
		\item For any $(\bomega, \ol \bomega) \in \bOmega_\bbG \times \bOmega_\bbG$ and $n \in \N$,
			\begin{equation}\label{eq:euclidean-cc-ineq}
				C_n(\bomega, \ol \bomega) \geq c_n(\pi_1 \bomega, \pi_1 \ol \bomega).
			\end{equation}
    \end{enumerate}
\end{lemma}

\begin{remark}
  We will show in Theorem~\ref{thm:t2-carnot-path},
  that the convergence in \eqref{eq:stability-euclidean} and the lower bound in  \eqref{eq:euclidean-cc-ineq} imply the lower estimate
  $\liminf_{n\to\infty}\T_{C_n,2}(\bmu,\bnu)\geq \T_{C_\cH,2}(\bmu,\bnu)$. An even stronger result will be derived in Section~\ref{sec:gamma-conv}, namely, we prove the $\Gamma$-convergence of the cost functions $C_n$ to the  cost $C_\cH$. This establishes that $C_\cH$
    is indeed the natural limiting cost. Moreover, we prove
    that the $\Gamma$-convergence of $C_n$ implies the convergence of the associated transport problems $\T_{C_n,2}(\bmu,\wt\bnu^n)$ along a suitable sequence of probability measures $\wt\bnu^n\in\Pc(\bOmega_\bbG)$.
\end{remark}

\begin{proof}
	Let $\omega, \ol \omega \in \Omega$ and write $h = \ol \omega - \omega$. The sequence $(c_n(\omega, \ol \omega))_{n \in \N}$ is increasing by definition. Suppose that $h \in \cH$. Then
	\begin{equation*}
		\lim_{n \to \infty} c_n^2(\omega, \ol \omega) = \lim_{n \to \infty} 2^n \sum_{k = 1}^{2^n}|h_{t_{k - 1}^n, t_k^n}|^2 = \int_0^1 |\dot h_t|^2 \di t = \|h\|_\cH^2.
	\end{equation*}
	For $h \notin \cH$, the above limit is $+ \infty$. This proves part (i).
	
	For part (ii), note that $\T_{c_{\cH}, 2}(\mu, \nu)$ and $\T_{c_n, 2}(\mu, \nu)$ admit minimisers, for each $n \in \N$, by e.g.\ \cite[Theorem 4.1]{Vi09}, since the cost functions are lower semicontinuous and non-negative. Let $\lambda^\ast \in \Pi(\mu, \nu)$ attain the infimum in $\T_{c_{\cH}, 2}(\mu, \nu)$. Then, by the monotone convergence theorem,
	\begin{equation*}
		\limsup_{n \to \infty}\T_{c_n, 2}^2(\mu, \nu) \leq \limsup_{n \to \infty}\int_{\Omega \times \Omega} c_{n}^2(\omega, \ol \omega) \di \lambda^\ast(\omega, \ol \omega) = \int_{\Omega \times \Omega} c_{\cH}^2(\omega, \ol \omega) \di \lambda^\ast(\omega, \ol \omega) = \T_{c_\cH, 2}^2(\mu, \nu).
	\end{equation*}
	On the other hand, for each $n \in \N$, let $\lambda^n \in \Pi(\mu, \nu)$ attain the infimum in $\T_{c_n, 2}(\mu, \nu)$. Since $\Pi(\mu, \nu)$ is tight, Prohorov's theorem implies that $(\lambda^n)_{n \in \N}$ converges weakly along a subsequence $(n_k)_{k \in \N}$ to some $\wt \lambda \in \Pi(\mu, \nu)$. By monotonicity, for any $m \in \N$, 
	\begin{align*}
		\liminf_{k \to \infty}\T_{c_{n_k}, 2}^2(\mu, \nu) = \liminf_{k \to \infty}\int_{\Omega \times \Omega} c_{n_k}^2(\omega, \ol \omega) \di \lambda^{n_k}(\omega, \ol \omega) & \geq \liminf_{k \to \infty}\int_{\Omega \times \Omega}c_{m}^2(\omega, \ol \omega) \di \lambda^{n_k}(\omega, \ol \omega) \\ 
		& = \int_{\Omega \times \Omega} c_{m}^2(\omega, \ol \omega) \di \wt \lambda(\omega, \ol \omega).
	\end{align*}
	Applying monotone convergence once more,
	\begin{align*}
		\liminf_{k \to \infty}\T_{c_{n_k}, 2}^2(\mu, \nu) \geq \lim_{m \to \infty}\int_{\Omega \times \Omega} c_{m}^2(\omega, \ol \omega) \di \wt \lambda(\omega, \ol \omega) & = \int_{\Omega \times \Omega} c_{\cH}^2(\omega, \ol \omega) \di \wt \lambda(\omega, \ol \omega)\\
		& \geq \T_{c_\cH, 2}^2(\mu, \nu).
	\end{align*}
	
	Now observe that, for $\bx, \ol\bx \in \bbG$, $d_\cc(\bx, \ol \bx) \geq |\pi_1 \bx - \pi_1 \ol\bx|$. Indeed, by definition of the Carnot--Carath\'eodory metric,
    \begin{align*}
        d_\cc(\bx, \ol\bx)& = \min\Big\{\, \int_0^1|\dot \gamma_t| \di t : \gamma \colon [0, 1] \to \bbG \; \text{horizontal,} \; \gamma_0 = \bx, \gamma_1 = \ol\bx \, \Big \}\\
        & \geq  \inf\Big\{\, \int_0^1|\dot \gamma_t| \di t : \gamma \colon [0, 1] \to \bbG \; \text{horizontal,} \; \pi_1\gamma_0 = \pi_1\bx, \pi_1\gamma_1 = \pi_1\ol\bx \, \Big \}\\
        & = \min \Big\{\, \int_0^1 |\dot g_t| \di t : g \colon [0, 1] \to \R^d \; \text{absolutely continuous,} \; \; g_0 = \pi_1\bx, g_1 = \pi_1\ol\bx \, \Big\}\\
        & = |\pi_1 \bx - \pi_1 \ol\bx|.
    \end{align*}
	Hence, for $\bomega, \ol \bomega \in \bOmega_\bbG$,
	\begin{align*}
		C_n^2(\bomega, \ol \bomega) & = 2^n \sum_{k = 1}^{2^n} d_\cc^2(\bomega_{t_{k - 1}^n, t_k^n}, \ol \bomega_{t_{k - 1}^n, t_k^n}) \geq 2^n \sum_{k = 1}^{2^n} |\pi_1 \ol \bomega_{t_{k - 1}^n, t_k^n} - \pi_1 \bomega_{t_{k - 1}^n, t_k^n}|^2 = c_n^2(\pi_1 \bomega, \pi_1 \ol \bomega). 
	\end{align*}
	This concludes part (iii).
\end{proof}
We are now in position to prove the main result of this section.
\begin{theorem}\label{thm:t2-carnot-path}
	Let $\bbG$ be a step-$2$ Carnot group and suppose that $\bmu_1 \in \Tc_2(\bbG, d_\cc, \alpha)$, for some $\alpha \in (0, \infty)$. Then $\bmu \in \Tc_2(\bOmega_\bbG, C_\cH, \alpha)$.
\end{theorem}

\begin{proof}
     If $\bnu$ is not absolutely continuous with respect to $\bmu$, then $H(\bnu \| \bmu) = + \infty$ and the cost-information inequality holds trivially.
     
    Now suppose that $\bnu \ll \bmu$. Recall the lift map $
    \Psi \colon C([0, 1], \R^d) \to C([0, 1], \bbG)$ given by \cref{eq:lift}. By \cref{prop:BM-lift}, $\bmu = \Psi_\sharp \mu$, and by \Cref{lem:abs-cont-lift}, there exists $\nu \ll \mu$ such that $\bnu = \Psi_{\sharp} \nu$. Now let $\lambda^\ast \in \Pi(\mu, \nu)$ be such that
    \begin{equation*}
    	\T_{c_\cH, 2}^2(\mu, \nu) = \int_{\Omega \times \Omega}c_\cH^2 (\omega, \ol \omega) \di \lambda^\ast(\omega, \ol \omega) = \int_{\Omega \times \Omega} \|\ol \omega - \omega\|_\cH^2 \di \lambda^\ast(\omega, \ol \omega).
    \end{equation*}
    We have that $\lambda^\ast(\{\, (\omega, \ol \omega) \in \Omega \times \Omega : \ol \omega - \omega \in \cH \,\}) = 1$ and that
    $\wt \blambda = (\Psi \times \Psi)_{\sharp} \lambda^\ast \in \Pi(\bmu, \bnu)$ is an admissible coupling.
    Using the property \cref{eq:lift-shift} of the lift and shift from \cref{prop:shift}, we find that
    \begin{align*}
    	\int_{\bOmega_\bbG \times \bOmega_\bbG}C_\cH^2(\bomega, \ol \bomega) \di \wt \blambda(\bomega, \ol \bomega) & = \int_{\Omega \times \Omega} C_\cH^2 (\Psi (\omega), \Psi (\ol \omega))\di \lambda^\ast(\omega, \ol \omega)\\
    	& = \int_{\Omega \times \Omega} C_\cH^2 (\Psi (\omega), T_{\ol \omega - \omega}\Psi(\omega))\di \lambda^\ast(\omega, \ol \omega)\\
    	& = \int_{\Omega \times \Omega} \|\ol \omega - \omega\|_\cH^2 \di \lambda^\ast(\omega, \ol \omega) = \T_{c_\cH, 2}^2(\mu, \nu).
    \end{align*}
    Hence $\T_{C_\cH, 2}^2(\bmu, \bnu) \leq \T_{c_\cH, 2}^2(\mu, \nu)$.
    
    Combining \cref{eq:stability-euclidean} and \cref{eq:euclidean-cc-ineq} from \cref{lem:euclidean-liminf}, we have
    \begin{align*}
    	\limsup_{n \to \infty} \T_{C_n, 2}^2(\bmu, \bnu) & \geq \limsup_{n \to \infty} \T_{c_n, 2}^2(\mu, \nu) \geq \T_{c_\cH, 2}^2(\mu, \nu) \geq \T_{C_\cH, 2}^2(\bmu, \bnu).
    \end{align*}
    By \cref{lem:rel-entropy-bound}, we conclude that $
    	\T_{C_\cH, 2}^2(\bmu, \bnu) \leq \frac{2}{\alpha} H(\bnu \| \bmu)$.
\end{proof}

\begin{theorem}\label{thm:t2-rp-space}
    Let $\bbG$ be an H-type group. Then the measure $\bmu$ on the space $\bOmega_\bbG = C_0([0,1], \bbG)$ satisfies the cost-information inequality
    \begin{equation*}
        \bmu \in \Tc_2(\bOmega_\bbG, C_\cH, \alpha),
    \end{equation*}
    for $\alpha > 0$ as in \cref{thm:t2-h-type}.
\end{theorem}

\begin{proof}
	By \cref{thm:t2-h-type}, $\bmu_1 \in \Tc_2(\bbG, d_\cc, \alpha)$. We conclude by \cref{thm:t2-carnot-path}.
\end{proof}

\subsection{Failure of top-down projection and blow-up of cost functions}\label{sec:projection-blow-up}

In this section, we point out two major differences between the classical Euclidean and the Carnot group settings. In contrast to the Euclidean case, we cannot project the Talagrand inequality from \cref{thm:t2-carnot-path} down to a Talagrand inequality for $\boldsymbol{B}_{1}$. Moreover, the cost functions $C_n$ do not converge pointwise to the cost $C_\cH$.

We start by giving the corresponding projection result in the Euclidean setting, which we prove via the contraction principle from \cref{lem:contraction}.
Let $\wt{P}_{1} \colon \Omega \to\R^{d_1},\, \omega \mapsto \omega_1$ denote the map that evaluates a path at time $t = 1$, and recall the Euclidean Cameron--Martin cost $c_\cH$ defined in \eqref{eq:cameron-martin}.

\begin{proposition}
	Let $\eta$ be a Borel probability measure on $\Omega$ and suppose that $\eta \in \Tc_2(\Omega, c_\cH, \alpha)$, for some $\alpha \in (0, \infty)$. Then $(\wt P_1)_\sharp\eta \in \Tc_2(\R^{d_1}, |. - .|, \alpha)$, where $|\cdot|$ denotes the Euclidean norm on $\R^{d_1}$.
\end{proposition}

\begin{proof}
    For any $\omega, \ol\omega \in \Omega$ such that $\omega - \ol\omega \in \cH$, Jensen's inequality implies that
    \[\|\omega - \ol \omega\|_{\infty}^2= \sup_{t\in[0,1]}\Big|\int_{0}^{t}(\dot{\omega}_s-\dot{\ol\omega}_s) \di s\Big|^2\leq \sup_{t\in[0,1]} t \int_{0}^{t}|\dot{\omega}_s-\dot{\ol\omega}_s|^2 \di s\leq \int_{0}^{1}|\dot{\omega}_s-\dot{\ol\omega}_s|^2 \di s=\|\omega-\ol\omega\|_{\cH}^2
    \]
    Thus, for any $\omega, \ol \omega \in \Omega$,
    \begin{align*}
        |\wt{P}_1\omega - \wt{P}_1\ol \omega|^2\leq \|\omega - \ol\omega\|_{\infty}^2\leq c_{\cH}^2(\omega,\ol\omega).
    \end{align*}
    Hence the contraction principle from \cref{lem:contraction} yields the claim.
\end{proof}

Now consider the law $\bmu$ of the Brownian motion $\bB$ on $\bbG$. Let $P_1 \colon \bOmega_\bbG \to \bbG$ denote the projection of a $\bbG$-valued path onto its final time evaluation; i.e.\ $P_1 \bomega = \bomega_1$ for any $\bomega \in \bbG$. Suppose that there exists $\wt\bOmega_\bbG \subseteq \bOmega_\bbG$ with $\bmu(\wt\bOmega_\bbG) = 1$ and some measurable function $L \colon \bOmega_\bbG \to [0, \infty]$ such that
\begin{align}\label{eq:proj-contraction}
	d_\cc(P_1 \bomega, P_1 \wt\bomega) \leq L(\bomega) C_\cH(\bomega, \wt\bomega), 	
\end{align}
for all $\bomega, \wt\bomega \in \wt\bOmega_\bbG$.
If $L \in L^\infty(\bmu)$, then, by \Cref{lem:contraction}, the $\Tc_2$ inequality for $\bmu$ implies a $\Tc_2$ inequality for $\bmu_1$. If we only have $L \in L^q(\bmu)$ for some $q \in [2, \infty)$, then \Cref{lem:contraction} still implies a $\Tc_p$ inequality for $p = \tfrac{2q}{2 + q} \in [1, 2)$. The following result shows that any such $L$ cannot belong to $L^q$ for any $q \in [2, \infty]$, and thus the contraction principle from \Cref{lem:contraction} is not applicable.

\begin{proposition}\label{prop:projection-fails}
	Let $L \colon \bOmega_\bbG \to [0, \infty]$ be as in \eqref{eq:proj-contraction}. Then $\bmu(L = \infty) > 0$. In particular, $L \notin L^q(\bmu)$ for any $q \in (0, \infty]$.
\end{proposition}

We make use of the following example in the proof of \cref{prop:projection-fails} and again below in the proof of \cref{prop:blow-up}.

\begin{lemma}\label{lem:shift-example}
	Let $a > 0$ and define $h \in \cH$ by $h_t = (a t, 0, \dotsc, 0) \in \R^{d_1}$, for all $t \in [0, 1]$. Then, for any $s, t \in [0, 1]$ with $s \leq t$, there exists a standard normal random variable $Z_{s, t}$ such that
	\begin{align*}
		d_\cc^2(\bB_{s, t}, (T_h \bB)_{s, t}) \geq a C (t - s)^{\frac{3}{2}} |Z_{s, t}|,
	\end{align*}
	for some constant $C > 0$ independent of $a$, $s$ and $t$. Moreover, for $u, v, s, t \in [0, 1]$ with $u \leq v \leq s \leq t$, the random variables $Z_{u,v}$ and $Z_{s, t}$ are independent.
\end{lemma}

\begin{proof}
	For $\bB = (\bB^{(1)}, \bB^{(2)})$, write $\bB^{(1)} = (B^1, \dotsc, B^{d_1})$. Let $s, t \in [0, 1]$ with $s \leq t$. Since $h$ is only non-zero in its first component, \Cref{rem:commutativity-error} implies that
	\begin{align*}
		\bB_{s, t}^{-1}(T_h \bB)_{s,t} & = \bigg(h_{s,t}, a \sum_{j = 2}^{d_1}\sco_{1, j}\int_s^t(s - r)\di B^j_r\bigg).
	\end{align*}
	Choose $k \in \{1, \dotsc, m\}$ such that $s_k^2 \coloneqq \sum_{j = 2}^{d_1} |\sco_{1,j}^k|^2 > 0$. By the left-invariance of $d_\cc$ and the estimate \eqref{eq:gauge-distance}, there exists a constant $\kappa > 0$ such that
	\begin{align*}
		d_\cc^2(\bB_{s, t}, (T_h \bB)_{s, t}) & \geq a \kappa^{-1} \bigg| \sum_{j = 2}^{d_1} \sco_{1, j} \int_s^t (s - r) \di B^j_r \bigg|
		\geq a \kappa^{-1} \bigg| \sum_{j = 2}^{d_1} \sco_{1, j}^k \int_s^t (s - r) \di B^j_r \bigg|.
	\end{align*}
	By It\^o's isometry, we can define a standard normal random variable
	\begin{align*}
		Z_{s, t} \coloneqq s_k^{-1} (t - s)^{-\frac{3}{2}} \sqrt 3 \sum_{j = 2}^{d_1} \sco_{1, j}^k \int_s^t (s - r) \di B^j_r.
	\end{align*}
	Thus, setting $C = 3^{-\frac{1}{2}} \kappa^{-1} s_k $, we have
	\begin{align*}
		d_\cc^2(\bB_{s, t}, (T_h \bB)_{s, t}) \geq a C (t - s)^{\frac{3}{2}} |Z_{s, t}|.
	\end{align*}
	The independence property follows from the independence of Brownian increments.
\end{proof}

\begin{proof}[Proof of \Cref{prop:projection-fails}]
	Suppose for contradiction that $\bmu(L < \infty) = 1$. Let $\bB$ be a Brownian motion on $\bbG$, let $\delta > 0$, and define $h \in \cH$ by $h_t = (\delta t, 0 \dotsc, 0) \in \R^{d_1}$, for all $t \in [0, 1]$. By \Cref{lem:shift-example}, there exists a constant $C > 0$ independent of $\delta$ and a standard normal random variable $Z$ such that we have the lower bound
	\begin{align*}
		d_\cc^2(P_1 \bB, P_1 (T_h \bB)) \geq \delta C |Z|.
	\end{align*}
	On the other hand, by definition of the cost $C_\cH$,
	\begin{align*}
		C_\cH^2(\bB, T_h \bB) = \|h\|_\cH^2 = \delta^2.
	\end{align*}
	Therefore, \eqref{eq:proj-contraction} implies that
	\begin{align*}
		C \delta |Z| \leq L(\bB) \delta^2.
	\end{align*}
	Since both $|Z|$ and $L(\bB)$ are almost surely finite, taking the limit as $\delta \to 0$ gives a contradiction.
\end{proof}

We now show that, contrary to the Euclidean case, the cost functions $C_n$ defined in \eqref{eq:defCostDisc} may not converge pointwise to $C_\cH$. Again, we use the example from \Cref{lem:shift-example}.

\begin{proposition}\label{prop:blow-up}
	Let $\bB$ be a Brownian motion on $\bbG$, and define $h \in \cH$ by $h_t = (t, 0, \dotsc, 0) \in \R^{d_1}$, for all $t \in [0, 1]$. Then
	\begin{equation*}
		\lim_{n \to \infty} C_n(\bB, T_h \bB) = \infty
	\end{equation*}
	almost surely.
\end{proposition}

\begin{proof}
	Fix $n \in \N$.
	By \Cref{lem:shift-example}, there exist independent standard normal random variables $Z_{n, k}$, $k \in \{1, \dotsc, 2^n\}$, such that
	\begin{align*}
		C_n^2(\bB, T_h \bB) & = 2^n \sum_{k = 1}^{2^n} d_\cc^2(\bB_{t^n_{k-1}, t^n_k}, (T_h \bB)_{t^n_{k - 1},t^n_k})
		\geq C 2^{-n/2} \sum_{k = 1}^{2^n}|Z_{n, k}|.
	\end{align*}
	Note that $(|Z_{n,k}|)_{k=1,\dots, 2^n}$ are independent half-normal random variables with mean $\sqrt{2/\pi}$ and variance $1 - 2/\pi$. 
    Thus, by Chebyshev's inequality, for any $\delta > 0$,
	\begin{align*}
		\P\bigg(\bigg\lvert\sum_{k = 1}^{2^n} \frac{|Z_{n,k}| - \sqrt{2/\pi}}{2^n}\bigg \rvert > \delta\bigg) \leq \delta^{-2}(1 - 2/\pi)2^{-n}.
	\end{align*}
	The right-hand side is summable in $n$ and so, by the first Borel--Cantelli lemma,
	\begin{equation*}\label{eq:LLN}
		\lim_{n \to \infty}2^{-n} \sum_{k = 1}^{2^n}|Z_{n, k}| = \sqrt{2/\pi} \quad \text{almost surely.}
	\end{equation*}
	Hence, we have
	\begin{align*}
		\liminf_{n \to \infty}C_n^2(\bB, T_h \bB)  \geq C \liminf_{n \to \infty}2^{n/2} \cdot 2^{-n} \sum_{k = 1}^{2^n}|Z_{n, k}| = C \sqrt{2/\pi} \lim_{n \to \infty} 2^{n/2} = + \infty,
	\end{align*}
	almost surely.
\end{proof}

When the marginals are related by a deterministic shift, we can identify the optimal coupling for the cost $C_\cH$, and we have the following equality.
\begin{lemma}\label{lem:h-determ-optim}
	Let $h \in \cH$ be deterministic and $\bnu = \Law(T_h \bB)$. Then $\T_{C_\cH, 2}(\bmu, \bnu)$ admits a unique optimal coupling, this coupling is induced by a Monge map, and
	\begin{align*}
		\T_{C_\cH, 2}^2(\bmu, \bnu) = \E[C_\cH^2(\bB, T_h \bB)] = 2 H(\bnu \| \bmu)<\infty.
	\end{align*}
\end{lemma}

\begin{proof}
	The second equality follows from \cref{thm:t2-top-foellmer}. Indeed, if $\bnu=\Law(T_h\bB)$, then $h=b^{\nu}$ for the Föllmer drift $b^{\nu}$ from \cref{thm:t2-top-foellmer}.
    Since $\bnu=\Law(T_h\bB)\ll \bmu$, we have that $H(\bnu\|\bmu)<\infty$.
	Now, from the definition of $C_\cH$, and the fact that $h \in \cH$ is deterministic, we have optimality of $\blambda^\ast = \Law(\bB, T_h \bB)$. 
	To see this, let $\blambda \in \Pi(\bmu, \bnu)$ and consider the event $E \coloneqq \{\,(\bomega, \ol \bomega) \in \bOmega_\bbG \times \bOmega_\bbG : \ol \bomega = T_{\pi_1 \ol \bomega - \pi_1\bomega}\bomega \,\}$. If $\blambda(E) < 1$, then by definition of $C_\cH$, we have $\int C_\cH^2(\bomega, \ol \bomega) \di \blambda (\bomega, \ol \bomega) = + \infty$. Suppose now that $\blambda(E) = 1$. Then, by Jensen's inequality,
	\begin{align*}
		\int C_\cH^2(\bomega, \ol \bomega)\di \blambda (\bomega, \ol \bomega) & = \int \|\pi_1 \ol \bomega - \pi_1 \bomega\|_\cH^2 \di \blambda(\bomega, \ol \bomega)\\
		& \geq \Bigl\| \int \pi_1 \ol \bomega \di \bnu(\ol \bomega) - \int \pi_1 \bomega \di \bmu(\bomega)\Bigr\|_\cH^2 = \|h\|_\cH^2.
	\end{align*}
	Equality holds if and only if $\blambda(\{\, (\bomega, \ol \bomega) \in \bOmega_\bbG \times \bOmega_\bbG : \pi_1 \ol \bomega - \pi_1 \bomega = h \,\}) = 1$. Combined with the condition that $\blambda(E) = 1$, we see that any optimal coupling is concentrated on the graph of the function $T_h \colon \bOmega_\bbG \to \bOmega_\bbG$. Thus, there is a unique optimal coupling of Monge form given by $\blambda^\ast = (\id \times T_h)_\sharp \bmu = \Law(\bB, T_h \bB)$, and $\T_{C_\cH, 2}^2(\bmu, \bnu) = \E[C_\cH^2(\bB, T_h \bB)]$.
\end{proof}

\begin{remark}
Let $h\in\mathcal{H}$ be as in \cref{prop:blow-up}. Since $\bnu=\Law(T_h\bB)\ll \bmu$, we have that $H(\bnu\|\bmu)<\infty$.
	By \cref{lem:rel-entropy-bound}, we thus observe that $\T_{C_n, 2}^2(\bmu, \bnu) \leq H(\bnu \| \bmu) <\infty$.
	However, \cref{prop:blow-up} shows that $\lim_{n \to \infty} C_n(\bB, T_h \bB) = + \infty$ almost surely. Thus $\Law(\bB, T_h \bB)$ is suboptimal for some $C_n$, $n \in \N$.
	
	In the case of the Heisenberg group $\bbG = \bbH^n$, this suboptimality can already be seen for $\T_{d_\cc, 2}(\bmu_1, \bnu_1)$. Indeed, \cite[Theorem 5.1]{AmRi04} shows that there is a unique optimal coupling and that this coupling is concentrated on the graph of some function $\phi \colon \bbH^n \to \bbH^n$. Taking, for example, $h$ as in \Cref{lem:shift-example}, it is clear that $\Law(\bB_1, (T_h \bB)_1)$ is not concentrated on any such graph, since $(T_h \bB)_1$ is not measurable with respect to $\sigma(\bB_1)$.
\end{remark}

\subsection{Riemannian approximation of the Heisenberg group}\label{sec:riemann}

In the case of classical Wiener space, we consider paths taking values in $\R^d$ with the Euclidean geometry. In the present Carnot group setting, we note the following two distinctions:
\begin{enumerate}[label = (\roman*)]
	\item $(\bbG, d_\cc)$ is a sub-Riemannian metric space (the sub-Laplacian is hypoelliptic);
	\item the group operation on $\bbG$ is non-commutative.
\end{enumerate}
It is shown in \cite[Theorem 2.12]{CaDaPaTy07} that any Carnot group can be approximated by Riemannian manifolds in the sense of pointed Gromov--Hausdorff convergence; see also \cite[Section 2.5]{AS2019HEFC} and, for the Heisenberg group, \cite[Section 6]{AmRi04}. In making this approximation, we move out of the sub-Riemannian setting but retain non-commutativity. We observe that, in this case, a $\Tc_2$ inequality on path space also holds (Proposition~\ref{prop:rie-t2-path}) 
and that the blow-up of discretised cost functions shown in \cref{prop:blow-up} does not occur (see Proposition~\ref{prop:RiemannBlowUp}). 
The failure to recover the $\Tc_2$ inequality on the underlying space via projection that was shown in \cref{prop:projection-fails} is still observed (Proposition~\ref{prop:RiemannContract}). 
However, in contrast to the sub-Riemannian case, we can use the contraction principle to obtain a $\Tc_p$ inequality on the underlying space for any $p \in [1, 2)$  (Proposition~\ref{prop:rie-t2-proj}).

In order to ease the presentation of this section, we specialise to the  Heisenberg group $\bbH = \bbH^1 \cong \R^2 \oplus \R$. Recall the left-invariant vector fields $(V_1, V_2, V_3) = (X, Y, Z)$, where
\begin{align*}
	X = \partial_{x}+\tfrac{1}{2}y\partial_{z}, \quad Y = \partial_{y}-\tfrac{1}{2}x\partial_{z}, \quad Z = [X, Y] = \partial_{z},
\end{align*}
and the group operation
\begin{equation*}
	\bx \bx^\prime = \big(x + x^\prime, y + y^\prime, z + z^\prime + \tfrac 12(x y^\prime {-} x^\prime y)\big), \quad \bx = (x, y, z), \, \bx^\prime = (x^\prime, y^\prime, z^\prime) \in \bbH.
\end{equation*}

For $\varepsilon > 0$, define the manifold $M_\varepsilon$ to be $\R^{2} \oplus \R$ equipped with the Euclidean topology and orthonormal basis $(X, Y, Z_\varepsilon)$, where $Z_\varepsilon = \varepsilon Z$. Let $d_\varepsilon$ denote the induced Riemannian distance, which is again left invariant. By \cite[Theorem 2.12]{CaDaPaTy07}, $(\bbH, d_\cc)$ is the limit of the Riemannian manifolds $(M_\varepsilon, d_\varepsilon)$ as $\varepsilon \to 0$, in the sense of pointed Gromov--Hausdorff convergence.
As in \cite[Section 6]{AmRi04}, we see that, for any $\bx, \by \in M_\varepsilon$,
\begin{align*}
	d_\varepsilon(\bx, \by) = \inf \Bigl\{\, \int_0^1 \sqrt{|\dot \gamma^1_t|^2 + |\dot \gamma^2_t|^2 + \varepsilon^{-2}|\dot \gamma^3_t - \tfrac{1}{2}(\gamma^1_t \dot \gamma^2_t - \dot \gamma^1_t \gamma^2_t)|^2}\di t : \gamma \in \AC([0, 1], M_\varepsilon), \, \gamma_0 = \bx, \gamma_1 = \by \,\Bigr\},
\end{align*}
and, for $\varepsilon_0, \varepsilon_1 > 0$ with $\varepsilon_1 \leq \varepsilon_0$,
\begin{align*}
	d_{\varepsilon_0}(\bx, \by) \leq d_{\varepsilon_1}(\bx, \by) \leq d_\cc(\bx, \by) = \sup_{\varepsilon > 0}d_\varepsilon(\bx, \by).
\end{align*}
Moreover, by \cite[Lemma 1.1]{Ju14}, there exists a constant $c > 0$ such that, for any $\varepsilon > 0$, $\bx, \by \in M_\varepsilon$,
\begin{align}\label{eq:est-juillet}
	d_\cc(\bx, \by) \leq d_\varepsilon(\bx, \by) + c \varepsilon.
\end{align}
We will also make use of the following bounds. There exists a constant $\ol\kappa > 0$ such that, for any $\varepsilon > 0$ and $\bx = (0, 0, z) \in M_\varepsilon$,
\begin{equation}\label{eq:rie-distance-estimates}
	\ol\kappa(|z|^{\frac 12} - \varepsilon) \leq  d_\varepsilon(0, \bx) \leq \epsilon^{-1} |z|,
\end{equation}
where the lower bound follows from \eqref{eq:gauge-distance} combined with \eqref{eq:est-juillet}, and the upper bound from considering the length of a purely vertical path.

On the space $(M_\varepsilon, d_\varepsilon)$, we consider the same non-commutative group law as on $\bbH$, but now the distance $d_\varepsilon$ is Riemannian.

Consider a Brownian motion $B$ on $\R^3$ with law $\mu$ and Cameron--Martin space $\cH = W^{1,2}_0([0, 1], \R^3)$.
We can define a Brownian motion $\bB^\varepsilon$ on $(M_\varepsilon, d_\varepsilon)$ by
\begin{align*}
	\D \bB^\varepsilon_t = X(\bB^\varepsilon) \di B^1_t + Y(\bB^\varepsilon) \di B^2_t + Z_\varepsilon \di B^3_t,
\end{align*}
and let $\bmu^\varepsilon = \Law(\bB^\varepsilon)$.
Explicitly, $\bB^{\epsilon}=(\bB^{\epsilon,(1)},\bB^{\epsilon,(2)})$ with
\begin{align*}
\D \bB^{\varepsilon,(1)}_t=\D (B^{1},B^{2})_t, \quad \D \bB^{\varepsilon,(2)}_t=\frac{1}{2}(B^2_t \di B^1_t - B^1_t \di B^2_t) + \epsilon\di B^3_t.
\end{align*}

Let $\Omega \coloneqq C_0([0, 1], \R^3)$ and $\bOmega^\varepsilon \coloneqq C_0([0, 1], M_\varepsilon)$, and define a map $\Psi^\varepsilon \colon \Omega \to \bOmega^\varepsilon$ by
\begin{align*}
	\Psi^\varepsilon(\omega) = (0, \varepsilon\omega^3) \Psi((\omega^1, \omega^2)), \quad \omega = (\omega^1, \omega^2, \omega^3) \in \Omega,
\end{align*}
where $\Psi$ is the lift map defined in \cref{def:lift}. Define its domain as $\dom(\Psi^\varepsilon) \coloneqq \{\, \omega = (\omega^1, \omega^2, \omega^3) \in \Omega : (\omega^1, \omega^2) \in \dom(\Psi) \,\}$. For $\omega$ absolutely continuous, $\Psi^\varepsilon$ takes the explicit form
\begin{align*}
	\Psi^\varepsilon(\omega)_t = \Big((\omega^1_t, \omega^2_t), \, \frac{1}{2}\int_0^t (\omega^1_r \di \omega^2_r - \omega^2_r \di \omega^1_r) + \varepsilon \omega^3_t\Big), \quad t \in [0, 1].
\end{align*}
Similarly to \cref{prop:BM-lift}, we have that $\bB^\varepsilon = \Psi^\varepsilon(B)$ almost surely. We can also define a shift map $T^\varepsilon_h \colon \bOmega^\varepsilon \to \bOmega^\varepsilon$, for any $h = (h^1, h^2, h^3) \in \cH$, by
\begin{align*}
	T^\varepsilon_h \bomega = (0, \varepsilon h^3)T_{(h^1, h^2)} \bomega, \quad \bomega \in \bOmega^\varepsilon.
\end{align*}
Then, similarly to \cref{prop:shift}, for any $\omega \in \dom(\Psi^\varepsilon)$ and $h \in \cH$, we have
\begin{align*}
	T^\varepsilon_h \Psi^\varepsilon(\omega) = \Psi^\varepsilon(\omega + h).
\end{align*}

Now define a cost function $C^\varepsilon_\cH \colon \bOmega^\varepsilon \times \bOmega^\varepsilon \to [0, \infty]$ by
\begin{align*}
	C^\varepsilon_\cH(\bomega, \wt\bomega) =
	\begin{cases}
		\|h\|_\cH, & \text{if} \; \wt\bomega = T^\varepsilon_h \bomega, \quad \text{for some}\; h \in \cH,\\
		+ \infty, & \text{otherwise}.
	\end{cases}
\end{align*}
Taking the same approach as in \cref{thm:top-t2-riedel-approach}, we see that $\bmu^\varepsilon$ satisfies a $\Tc_2$ inequality with this cost.
\begin{proposition}\label{prop:rie-t2-path}
	We have the cost-information inequality $\bmu^\varepsilon \in \Tc_2(\bOmega^\varepsilon, C^\varepsilon_\cH, 1)$.
\end{proposition}

\begin{proof}
	As noted in \cref{sec:prelim-t2}, $\mu \in \Tc_2(\Omega, c_\cH, 1)$, where the cost $c_\cH$ is defined in \cref{eq:BM-cost}. We also have that $\bmu^\varepsilon = \Psi^\varepsilon_\sharp \mu$ and $\mu(\dom(\Psi^\varepsilon)) = 1$. Moreover, for any $x, y \in \dom(\Psi^\varepsilon)$ with $h \coloneqq y - x \in \cH$,
	\begin{align*}
		C^\varepsilon_\cH(\Psi^\varepsilon(x), \Psi^\varepsilon(y)) & = C^\varepsilon_\cH(\Psi^\varepsilon(x), \Psi^\varepsilon(x + h)) = C^\varepsilon_\cH(\Psi^\varepsilon(x), T^\varepsilon_h\Psi^\varepsilon(x)) = \|h\|_\cH^2 = c_\cH(x, y).
	\end{align*}
	In case $y - x \notin \cH$, then both sides are infinite. Thus, applying the contraction principle from \cref{lem:contraction}, we have that $\bmu^\varepsilon \in \Tc_2(\bOmega^\varepsilon, C^\varepsilon_\cH, 1)$.
\end{proof}

In contrast to the sub-Riemannian setting, $\bmu^\varepsilon \in \Tc_2(\bOmega^\varepsilon, C^\varepsilon_\cH, 1)$ implies a $\Tc_p$ inequality for $\bmu^\varepsilon_1 \coloneqq \Law(\bB^\varepsilon_1)$, for $p \in [1, 2)$, with a constant depending on $\varepsilon$.

Let $P_1 \colon \bOmega^\varepsilon \to M_\varepsilon$ denote the projection $P_1(\bomega) = \bomega_1$, for $\bomega \in \bOmega^\varepsilon$, so that $\bmu^\varepsilon_1 = (P_1)_\sharp \bmu^\varepsilon$.

\begin{proposition}\label{prop:rie-t2-proj}
	Suppose that $\bmu^\varepsilon \in \Tc_2(\bOmega^\varepsilon, C^\varepsilon_\cH, 1)$. Then, for any $p \in [1, 2)$, there exists $\alpha(\varepsilon, p) > 0$ such that $\lim_{\varepsilon \to 0} \alpha(\varepsilon, p) = 0$ and $\bmu^\varepsilon_1 \in \Tc_p(M_\varepsilon, d_\varepsilon, \alpha(\varepsilon, p))$.
\end{proposition}

\begin{proof}
	Suppose that $\bomega, \wt\bomega \in \bOmega^\varepsilon$ with $\wt\bomega = T^\varepsilon_h \bomega$, for some $h \in \cH$. Define $\bh = \Psi^\varepsilon (h)$ and $\bgamma_1 = \bomega_1^{-1}(T^\varepsilon_h \bomega)_1$. Then
	\begin{align*}
		d_\varepsilon(P_1(\bomega), P_1(\wt\bomega)) & \leq d_\varepsilon(\bh_1, \bgamma_1) + d_\varepsilon(0, \bh_1).
	\end{align*}
	We bound $d_\varepsilon(0, \bh_1)$ by the length of the curve $t \mapsto \bh_t$ in $\bOmega^\varepsilon$ to get
	\begin{equation}\label{eq:rie-dist-bound}
	\begin{split}
		d_\varepsilon(0, \bh_1) & \leq \int_0^1 \sqrt{|\dot h^1_r|^2 + |\dot h^2_r|^2 + \varepsilon^{-2} |\varepsilon \dot h^3_r + \tfrac 12(h^1_r \dot h^2_r - h^2_r \dot h^1_r) - \tfrac 12(h^1_r \dot h^2_r - h^2_r \dot h^1_r)|^2} \di r\\
		& = \int_0^1 \sqrt{|\dot h^1_r|^2 + |\dot h^2_r|^2 + |\dot h^3_r|^2} \di r \leq \|h\|_\cH.
	\end{split}
	\end{equation}
	Similarly to \cref{rem:commutativity-error}, by integration by parts,
	\begin{align}\label{eq:rie-comm-error}
		\bh_1^{-1} \bgamma_1 = \Big(0, \int_0^1 \Big(h^1_{r} \di \omega^2_r - h^2_{r} \di \omega^1_r\Big)\Big),
	\end{align}
	and so, by \eqref{eq:rie-distance-estimates},
	\begin{align*}
		d_\varepsilon(\bh_1, \bgamma_1) & \leq \varepsilon^{-1} \Big|\int_0^1 \Big(h^1_{r} \di \omega^2_r - h^2_{r} \di \omega^1_r\Big)\Big|\\
		& \leq 2 \varepsilon^{-1} \|\omega\|_\infty \|h\|_\cH,
	\end{align*}
	where $\|\omega\|_\infty \coloneqq \sup_{t \in [0, 1]} |(\omega^1_t, \omega^2_t)|$. 
    Hence
	\begin{align*}
		d_\varepsilon(P_1(\bomega), P_1(\wt\bomega)) \leq (1 + 2 \varepsilon^{-1} \|\omega\|_\infty) \|h\|_\cH = (1 + 2 \varepsilon^{-1} \|\omega\|_\infty) C^\varepsilon_\cH(\bomega, \wt\bomega).
	\end{align*}
	In the case that there does not exist $h \in \cH$ such that $\wt\bomega = T_h \bomega$, then the same inequality holds trivially.
	
	Next note that, for any $q \in [1, \infty)$, $\bomega \mapsto \|\omega\|_\infty \in L^q(\bmu)$. Let $p \in [1, 2)$ and set $q = \tfrac{2p}{2 - p} \in [2, \infty)$. By the contraction principle from \Cref{lem:contraction}, we conclude that
	\begin{align*}
		\bmu_1 \in \Tc_p(M_\varepsilon, d_\varepsilon, \alpha(\varepsilon, p)),
	\quad\text{where}\quad
		\alpha(\varepsilon, p) = \bigl(1 + 2 \varepsilon^{-1} \E[\|B\|_\infty^q]^\frac{1}{q}\bigr)^{-2},
	\end{align*}
	and we see that $\lim_{\varepsilon \to 0} \alpha(\varepsilon, p) = 0$.
\end{proof}

Analogously to the sub-Riemannian setting, however, it is not possible to recover a $\Tc_2$ inequality via the contraction principle, as the next result shows.

\begin{proposition}\label{prop:RiemannContract}
	Let $\wt \bOmega \subseteq \bOmega^\varepsilon$ such that $\bmu(\wt\bOmega) = 1$, and let $L \colon \bOmega^\varepsilon \to [0, \infty]$ be a measurable function such that
	\begin{align*}
		d_\varepsilon(P_1 \bomega, P_1 \wt\bomega) \leq L(\bomega) C_\cH^\varepsilon(\bomega, \wt\bomega),
	\end{align*}
	for all $\bomega, \wt\bomega \in \wt\bOmega$. Then $L \notin L^\infty(\bmu^\varepsilon)$.
\end{proposition}

\begin{proof}
	Suppose for a contradiction that $L \in L^\infty(\bmu^\varepsilon)$. Define $h \in \cH$ by $h_t = (t, 0, 0) \in \R^3$, for all $t \in [0, 1]$. Then, similarly to \Cref{lem:shift-example}, we can apply \eqref{eq:rie-comm-error} and \eqref{eq:rie-distance-estimates} to see that there exists a standard normal random variable $Z$ and a constant $C(\varepsilon) > 0$ such that
	\begin{align*}
		d_\varepsilon^2(\bB^\varepsilon_1, (T^\varepsilon_h \bB^\varepsilon)_1) \geq C(\varepsilon) |Z|.
	\end{align*}
	We also have $C_\cH^\varepsilon (\bB^\varepsilon, T^\varepsilon_h \bB^\varepsilon)= \|h\|_\cH = 1$. Thus $C(\varepsilon)|Z| \leq L(\bB^\varepsilon)$. Since $|Z|$ is not essentially bounded, we arrive at a contradiction.
\end{proof}

Finally, in the Riemannian setting, we do not observe the blow-up shown in \Cref{prop:blow-up}. For $n \in \N$, define a cost $C^\varepsilon_n \colon \bOmega^\varepsilon \times \bOmega^\varepsilon \to [0, \infty)$ by
\begin{align*}
	C^\varepsilon_n(\bomega, \wt\bomega)^2 & = 2^n \sum_{k = 1}^{2^n}d_\varepsilon^2(\bomega_{t^n_{k - 1}, t^n_k}, \wt\bomega_{t^n_{k - 1}, t^n_k}), \quad \bomega, \wt \bomega \in \bOmega^\varepsilon.
\end{align*}

\begin{proposition}\label{prop:RiemannBlowUp}
	For any $h \in \cH$, we have
	\begin{align*}
		\limsup_{n \to \infty} C^\varepsilon_n(\bB^\varepsilon, T^\varepsilon_h \bB^\varepsilon) \leq C^\varepsilon_\cH(\bB^\varepsilon, T^\varepsilon_h \bB^\varepsilon) = \|h\|_\cH,
	\end{align*}
	almost surely.
\end{proposition}

\begin{proof}
	Define $\bZ$ by $\bZ_{s, t} = \bB_{s,t}^{-1}(T^\varepsilon_h \bB)_{s,t}$, for $s, t \in [0, 1]$ with $s \leq t$, and $\bh = \Psi^\varepsilon(h)$. Let $n \in \N$. By Young's inequality, we bound
	\begin{align}\label{eq:young-bound}
		C^\varepsilon_n(\bB^\varepsilon, T^\varepsilon_h \bB^\varepsilon)^2 \leq (1 + n) C^\varepsilon_n(\bh, \bZ)^2 + (1 + \tfrac{1}{n}) C^\varepsilon_n(0, \bh)^2.
	\end{align}
	As in \eqref{eq:rie-dist-bound}, we bound $d_\varepsilon(\bh_{t^n_{k-1}}, \bh_{t^n_k})$ by the length of the curve $\bh$; i.e.
	\begin{align*}
		d_\varepsilon(\bh_{t^n_{k-1}}, \bh_{t^n_k}) \leq \int_{t^n_{k-1}}^{t^n_k} \sqrt{|\dot h^1_r|^2 + |\dot h^2_r|^2 + |\dot h^3_r|^2} \di r.
	\end{align*}
	Applying the Cauchy--Schwarz inequality, we have
	\begin{align*}
		C^\varepsilon_n(0, \bh)^2 = 2^n \sum_{k = 1}^{2^n}d_\varepsilon^2(\bh_{t^n_{k-1}}, \bh_{t^n_k}) \leq \sum_{k = 1}^{2^n}\int_{t^n_{k-1}}^{t^n_k} \Bigl(|\dot h^1_r|^2 + |\dot h^2_r|^2 + |\dot h^3_r|^2\Bigr) \D r = \|h\|_\cH^2,
	\end{align*}
	and so
	\begin{align*}
		\limsup_{n \to \infty} (1 + \tfrac{1}{n})C^\varepsilon_n(0, \bh)^2 \leq \|h\|_\cH^2.
	\end{align*}
	
	Next, similarly to \eqref{eq:rie-comm-error} and \cref{rem:commutativity-error}, we have
	\begin{align*}
		\bh_{s,t}^{-1} \bZ_{s,t} = \Big(0, \int_s^t \Big(h^1_{s,r} \di B^2_r - h^2_{s,r} \di B^1_r\Big)\Big).
	\end{align*}
	Therefore, using the estimate \eqref{eq:rie-distance-estimates} and the fact that $B$ is almost surely $\beta$-H\"older continuous for any $\beta \in (0, 1/2)$, there exists a constant $c > 0$ such that we have the almost sure bound
	\begin{align*}
		d_\varepsilon^2(0, \bh_{s,t}^{-1} \bZ_{s,t}) & \leq \varepsilon^{-2}\Big|\int_s^t h^1_{s, r} \di B^2_r \Big|^2 + \varepsilon^{-2}\Big|\int_s^t h^2_{s, r} \di B^1_r \Big|^2\\
		& \leq c \varepsilon^{-2} \|B\|_\beta^2|t - s|^{1 + 2\beta} \int_s^t |\dot h_r|^2 \di r,
	\end{align*}
	where $\|B\|_\beta$ is the $\beta$-H\"older norm of $B$. Hence
	\begin{align*}
		C^\varepsilon_n(\bh, \bZ)^2 & = 2^n \sum_{k = 1}^{2^n}d_\varepsilon^2(0, \bh_{t^n_{k-1}, t^n_k}^{-1}\bZ_{t^n_{k-1}, t^n_k}) \leq c \varepsilon^{-2} 2^n 2^{-(1 + 2 \beta)n}\|h\|_\cH^2 \|B\|_\beta^2 = c \varepsilon^{-2} \|h\|_\cH^2 \|B\|_\beta^2 2^{- 2\beta n},
	\end{align*}
	and so $\lim_{n \to \infty} (1 + n)C^\varepsilon_n(\bh, \bZ)^2 = 0$ almost surely. We conclude by \Cref{eq:young-bound}.
\end{proof}

\subsection{\texorpdfstring{$\Gamma$}{Gamma}-convergence of the cost functions}\label{sec:gamma-conv}

Despite the pointwise blow-up of the cost functions $C_n$ that we demonstrated in \cref{prop:blow-up}, we now show that 
$C_n$  does converge to  $C_\cH$  in a variational sense. More precisely, the sequence $C_n$ converges to $C_\cH$ in the sense of $\Gamma$-convergence,   a notion of convergence for families of minimisation problems that is formulated in terms of asymptotic lower and upper bounds. 
On a metric space $(E,d_E)$, we say that a sequence of functionals $F_n \colon E\to \R\cup\{\infty\}$ $\Gamma$-converges to a limit $F_\infty \colon E\to \R\cup\{\infty\}$ if
\begin{enumerate}[label = (\roman*)]
	\item for every sequence $x_n \to x$ in $E$, we have $F_\infty(x) \leq\liminf_{n\to\infty} F_n(x_n)$; and
	\item for every $x \in E$, there exists a sequence $\wt x_n\to x$ in $E$ such that $\limsup_{n\to\infty} F_n(\wt x_n) \leq F_\infty(x)$.
\end{enumerate}
The sequence $(\wt x_n)$ in condition (ii) is called a \emph{recovery sequence} as it ``recovers'' the correct energy level 
$F_\infty(x)$ from the approximating energies $F_n(\wt x_n)$ by adding suitable oscillations to $x$. One may view 
$\Gamma$-convergence as describing the asymptotic behavior of energy landscapes, in close analogy with large deviation principles, which characterise the asymptotics of probability measures via their rate functions. Indeed, the interplay between $\Gamma$-convergence and large deviation principles has been studied in several publications; see e.g.\ \cite{Mari2012Gamma,Berm2018LDGM}.
A central advantage of  $\Gamma$-convergence is its stability property: convergence of functionals implies convergence of minimal values and, under mild compactness assumptions, convergence of (almost) minimisers. For a comprehensive treatment we refer to the monographs \cite{D1993IC,B2002GB,R2018CV}.

We will see that the $\Gamma$-convergence of the cost functions $C_n$ implies the convergence of the associated optimal transport problems; i.e.\ for every $\bnu\in \Pc(\bOmega_\bbG)$ there
exists a sequence of probability measures $\wt\bnu^n\in\Pc(\bOmega_\bbG)$ such
that the associated transport costs
$\T_{C_n, 2}(\wt\bnu^n,\bmu)$ converges to
$\T_{C_\cH, 2}(\bnu,\bmu)$ as $n\to\infty$.

In the Euclidean case with cost $c_n$ defined as in 
\eqref{eq:cost-euclidean}, we have for $\omega,\wt\omega\in C_0([0,1];\R^d)$ with $h = \omega-\wt \omega$ the 
 formula $c_n(\omega,\wt\omega) = \int_0^1 |\partial_s \wh h^n|^2\di s$,
where $\wh h^n$ is the piecewise affine interpolant for $h$.
One readily verifies that  $c_n(\omega,\wt\omega)$ converges
to $\|h\|_\cH$ if $h\in\cH$ and to $+\infty$ otherwise; cf.\ \Cref{lem:euclidean-liminf}. In particular, the $\Gamma$-convergence of $c_n$ also holds in this setting with the same limiting cost.

Recall that we consider the metric space $(\bOmega_\bbG, d_\infty)$, where $d_\infty$ is the uniform metric defined in \eqref{eq:unif-metric}.

We start with the following lemma, which gives the pointwise convergence 
of the cost to the Cameron--Martin norm for horizontal curves.

\begin{lemma}\label{lem:lifted-CM-path}
Let $h \in \cH$, and let $\bh = \Psi(h)$ denote its lift to $\bOmega_\bbG$. Then, for the family of cost functions $(C_n)_{n \in \N}$ defined in \eqref{eq:defCostDisc}, we have $C_n(0,\bh)\leq \|h\|_\cH$ 
and $\lim_{n \to \infty} C_n(0, \bh) = \|h\|_\cH$.
\end{lemma}

\begin{proof}
Let $h\in\cH$ and let $\bh = \Psi(h)$.
The curve $t \mapsto \bh_t\in\bbG$ is horizontal and therefore, for every $0\leq s<t\leq 1$, we have
\[
d_\cc(\bh_s, \bh_t) \leq \int_s^t |\dot h_r| \di r.
\]
Applying Hölder's inequality, we obtain the estimate
\[
C_n^2(0,\bh)= 2^n\sum_{i=1}^{2^n} d_{\cc}(\bh_{t_i^n},\bh_{t_{i-1}^n})^2
\leq \sum_{i=1}^{2^n} \int_{t_{i-1}^n}^{t_i^n}|\dot h_r|^2 \di r
= \|h\|_{\cH}^2.
\]
Taking the $\limsup$ on the left-hand side 
gives $\limsup_{n \to \infty} C_n(0, \bh) \leq \|h\|_\cH$. 

To show the lower bound, define a piecewise constant function $g^n \colon [0, 1] \to [0, \infty)$, for each $n \in \N$, by
$g^n_t = 2^n d_\cc(\bh_{t_i^n},\bh_{t_{i-1}^n})$
for $t\in [t_{i-1}^n,t_i^n)$, $i \in \{1,\dotsc, n\}$.
Note that the sequence $(g^n)$ is uniformly bounded in $L^2([0,1])$. 
Hence, we can extract a weakly converging subsequence such that $g^{n_k} \to g$
in $L^2([0,1])$. For given $0\leq r<s\leq 1$, we can find indices $i_n,\, j_n\in\{1,\ldots, 2^n\}$ such that,
for $r_n = t_{i_n}^n$ and $s_n = t_{j_n}^n$,
\[
0\leq r_n \leq r < s\leq s_n\leq 1,\quad\text{and}\quad
\lim_{n\to\infty} r_n = r,\quad
\lim_{n\to\infty} s_n = s.
\]
By the triangle inequality and the continuity of $t \mapsto \bh_t$, there is a sequence $(\varepsilon_n) \subset \R$ such that $\varepsilon_n \to 0$ and
\[
d_\cc (\bh_r,\bh_s) 
\leq \sum_{i=i_n}^{j_n}
d_\cc(\bh_{t_i^n},\bh_{t_{i-1}^n}) +d_\cc(\bh_r,\bh_{r_n})  +d_\cc(\bh_s,\bh_{s_n}) 
= \int_{r_n}
^{s_n} g^n_t \di t + \epsilon_n.
\]
Passing to the limit as $n\to\infty$, we infer that $d_\cc (\bh_r,\bh_s)\leq \int_r^s g_t \di t$.
Now, by the minimality of the metric derivative (see \Cref{rem:metric-derivative}), 
we obtain $g_t \geq |\dot h_t|$ for almost every $t\in(0,1)$.
Finally, since $C_n^2 (0,\bh) = 2^n\sum_{i=1}^{2^n} d_{\cc}^2(\bh_{t_i^n},\bh_{t_{i-1}^n})$, lower-semicontinuity of the norm implies that
\[
\liminf_{n\to \infty}C_n^2(0, \bh) = \liminf_{n\to \infty}\int_0^1 (g^n_t)^2\di t \geq \int_0^1 g_t^2 \di t \geq \|h\|^2_\cH,
\] 
which finishes the proof.
\end{proof}

The next lemma shows that the cost $C_n$ blows up along sequences $\bomega^n,\wt \bomega^n$ that converge to limits $\bomega,\wt \bomega$
whose difference $\bomega^{-1}\wt\bomega$ is a purely vertical process; i.e.\ $t \mapsto (\bomega^{-1}\wt\bomega)_t = (0, \theta_t)$.

\begin{lemma}\label{lem:vertical-cost}
	Let $\bomega, \wt \bomega \in \bOmega_\bbG$, and let $(\bomega^n), (\wt \bomega^n) \subset \bOmega_\bbG$ be sequences such that $\lim_{n \to \infty}(\bomega^n, \wt \bomega^n) = (\bomega, \wt \bomega)$. 
    Suppose that there exists a non-zero $\btheta \in \bOmega_\bbG$ such that $\btheta_t = (0, \theta_t)$ and $\wt \bomega_t = \bomega_t \btheta_t$ for all $t \in [0, T]$.
    Then  $\lim_{n \to \infty}C_n(\bomega^n, \wt \bomega^n) = \infty$.
\end{lemma}

\begin{proof}
	Suppose for contradiction that $C_n(\bomega^n, \wt \bomega^n)$ is bounded uniformly in $n$. 
    We use the equivalence of the gauge distance $d_g$ and the Carnot--Carath\'eodory distance $d_\cc$ from \eqref{eq:gauge-distance} to obtain the lower bound
    \[
    C_n(\bomega^n, \wt \bomega^n)^2  = 2^n\sum_{k=1}^{2^n} d_\cc\big(\bomega^n_{t_{k-1}^n,t_k^n}, \wt \bomega^n_{t_{k-1}^n,t_k^n}\big)^2 \geq
    \frac{2^n}{\kappa} \sum_{k=1}^{2^n}
      d_g\big(\bomega^n_{t_{k-1}^n,t_k^n}, \wt \bomega^n_{t_{k-1}^n,t_k^n}\big)^2\geq
      \frac{2^n}{\kappa} \sum_{k=1}^{2^n}
      \big|\theta^n_{t_{k-1}^n,t_k^n}\big|,
    \]
    where $\theta^n_{s,t} = \pi_2((\bomega_{s,t}^n)^{-1}\wt \bomega_{s,t}^n)$ are the increments of the vertical process. 
    Let $\wh{ \theta}^n_t$ denote the piecewise affine interpolant associated with the increments $\theta^n_{t_{k-1}^n,t_k^n}$ with $\wh\theta^n_0 = 0$.
    The above estimate gives the bound $C_n(\bomega^n, \wt \bomega^n)^2\geq \frac{2^n}{\kappa}\|{\wh \theta^n}\|_{W^{1,1}}$, 
    where the left-hand side is uniformly bounded with respect to $n$ by assumption.
    Thus ${\wh \theta^n}\to 0$ in $W^{1,1}_0([0,1],\R^{d_2})$ and so $\theta \equiv 0$ as $W^{1,1}_0([0,1],\R^{d_2})\hookrightarrow C_0([0,1], \R^{d_2})$, giving a contradiction. We conclude that $\lim_{n \to \infty}C_n(\bomega^n, \wt \bomega^n) = \infty$.
\end{proof}

Now we prove condition (ii) in the definition of $\Gamma$-convergence for $C_n$, i.e.\ the existence of a recovery sequence. In fact, we show a stronger version; see Remark~\ref{rem:recoveryMap}.

\begin{proposition}\label{prop:limsup}
	Let $\bomega, \wt \bomega \in \bOmega_\bbG$ and let $(\bomega^n)\subset \bOmega_\bbG$ be a sequence such that $\bomega^n \to \bomega$. Then there exists a sequence $(\wt\bomega^n)\subset\bOmega_\bbG$ such that $\wt\bomega^n\to \wt\bomega$ and
	\begin{equation}
    \label{eq:limsupCost}
		\limsup_{n \to \infty} C_n(\bomega^n, \wt \bomega^n) \leq C_\cH(\bomega, \wt \bomega).
	\end{equation}
\end{proposition}

\begin{proof}
We only have to consider the case $\wt \bomega= T_h \bomega$ for $h\in \cH$, since  the right-hand side in \eqref{eq:limsupCost}
is otherwise infinite by definition of $C_\cH$ in \eqref{eq:rp-cost}, and the inequality holds trivially.
Let $\bomega \in \bOmega_\bbG$, $h \in \cH$, and $\wt \bomega = T_h \bomega$. Consider a sequence $(\bomega^n)\subset\bOmega_\bbG$ such that $\bomega^n \to \bomega$ in $\bOmega_\bbG$. For each $n \in \N$, define $\ol\bomega^n:=T_h\bomega^n\in\bOmega_\bbG$.
By the continuity of the shift map from \Cref{prop:shift}, $\ol\bomega^n\to \wt \bomega = T_h \bomega$ in $\bOmega_\bbG$.
As in Remark~\ref{rem:commutativity-error}, we introduce the non-commutativity error
\[
\begin{split}
\btheta^n_{s,t}&:=\bh^{-1}_{s, t} (\bomega_{s,t}^n)^{-1}\ol\bomega_{s,t}^n 
=
\big(0, \theta^n_{s,t}\big),\\
\text{where}\quad \theta^n_{s,t}&= \Big(\int_s^t \SCO h_{s, r} \otimes \D \omega^{n}_r\Big).
\end{split}
\]
We define $\bvtheta^n \in \bOmega_\bbG$ such that its increments satisfy
$(\bvtheta_{t_{k-1}^n}^n)^{-1} \bvtheta_{t_k^n}^n=\btheta^n_{t_{k-1}^n,t_k^n}$.
Indeed, we set
$\bvtheta^n_t = (0,\vartheta^n_t)$, where
\[
	\vartheta^n_0 = 0 \quad \text{and} \quad
	\vartheta^n_t = \vartheta^n_{t_{k-1}^n}
	+ \theta^n_{t_{k-1}^n,t}, \quad \text{for} \quad t\in(t_{k-1}^n,t_k^n],\; k \in \{1,\dotsc,  2^n\}.
\]
We emphasise that, for any $t \in [0, 1]$, $\bvtheta^n_t$
is an element in the centre of the group $\bbG$. In particular, it commutes with every element in 
$\bbG$. Therefore, defining the curve $t \mapsto \wt \bomega^n_t = \ol \bomega^n_t(\bvtheta^n_t)^{-1}\in\bbG$, we obtain that its increments satisfy
\[
\wt\bomega^n_{s,t} = \ol\bomega^n_{s,t} (\bvtheta^n_{s,t})^{-1}
=(\bvtheta_{s,t}^n)^{-1}\ol\bomega_{s,t}^n, \quad 0 \leq s \leq t \leq 1.
\]
Using the left-invariance of $d_\cc$ and the definition of $\bvtheta^n$, 
we find that, for $k \in \{1,\dotsc,  2^n\}$,
\begin{align*}
d_\cc(\bomega^n_{t_{k-1}^n,t_{k}^n},\wt\bomega^n_{t_{k-1}^n,t_{k}^n})
& = d_\cc\big(0,(\bomega^n_{t_{k-1}^n,t_{k}^n})^{-1}\ol\bomega^n_{t_{k-1}^n,t_{k}^n}(\bvtheta^n_{t_{k-1}^n,t_{k}^n})^{-1}\big)\\
& = d_\cc\big(0,\bh_{t_{k-1}^n,t_{k}^n}\btheta^n_{t_{k-1}^n,t_{k}^n}(\btheta^n_{t_{k-1}^n,t_{k}^n})^{-1}\big)
= d_\cc(0,\bh_{t_{k-1}^n,t_{k}^n}).
\end{align*}
\Cref{lem:lifted-CM-path} now allows us to pass to the $\limsup$. More precisely, we have that
\[
\limsup_{n\to \infty} C_n(\bomega^n, \wt \bomega^n)
=
\limsup_{n\to \infty}C_n(0,\bh)\leq C_\cH(\bomega,\wt\bomega).
\]

It remains to show that $\wt \bomega^n \to\wt \bomega$. Due to the convergence $\ol \bomega^n \to \wt\bomega$ it suffices to show that $\bvtheta^n\to 0$ in $\bOmega_\bbG$ or, equivalently, $\vartheta^n\to 0$ in $C([0,1], \R^{d_2})$. Using the definition of $\vartheta^n$, we find a constant $C>0$ such that
\[
\|\vartheta^n\|_\infty
\leq C \sup_{k=1,\ldots, 2^n}
\max_{i<j}\sup_{t\in[t_{k-1}^n,t_k^n]}|\omega_{t}^{n,i}{-}\omega_{t_k^n}^{n,i}| \big\|\dot h^j\big\|_{L^1}.
\]
Since $\bomega^n$ converges uniformly to $\bomega$, the right-hand side vanishes as $n\to \infty$.
\end{proof}

We now prove the lower estimate that is required in condition (i) of the definition of $\Gamma$-convergence.

\begin{proposition}\label{prop:liminf}
	Let $\bomega, \wt \bomega \in \bOmega_\bbG$, and let $(\bomega^n), (\wt \bomega^n) \subset \bOmega_\bbG$ be sequences such that $\lim_{n \to \infty}(\bomega^n, \wt \bomega^n) = (\bomega, \wt \bomega)$.
	Then $\liminf_{n \to \infty} C_n(\bomega^n, \wt \bomega^n) \geq C_\cH(\bomega, \wt \bomega)$.
\end{proposition}

\begin{proof}
Consider a pair of curves $(\bomega,\wt{\bomega}) \in \bOmega_\bbG
\times\bOmega_\bbG$ and a pair of sequences $(\bomega^n, \wt \bomega^n) \subset \bOmega_\bbG \times \bOmega_\bbG$ such that $\lim_{n \to \infty}(\bomega^n, \wt \bomega^n) = (\bomega, \wt \bomega)$.
We will consider three cases.

\emph{Case 1a.}
First, we consider the case that $\wt\bomega = T_h\bomega$ for some $h\in \cH$.
We may assume that $I \coloneqq \liminf_{n\to \infty}C_n(\bomega^n, \wt \bomega^n)<\infty$, since otherwise the inequality holds trivially.
Let $h^n = \pi_1\wt \bomega^n - \pi_1\bomega^n$
so that $h^n\to h$
in $C([0,1], \R^{d_1})$.
We now apply \Cref{lem:euclidean-liminf} (iii), to see that
\begin{align}\label{eq:hn-ineq}
    \liminf_{n\to\infty}
    C_n(\wt\bomega^{n},\bomega^n)^2\geq 
    \liminf_{n\to\infty}
    2^n \sum_{k=1}^{2^n}|h_{t_k^n}^n-h_{t_{k-1}^n}^n|^2
    = \liminf_{n \to \infty}\|\wh{h}^n\|_\cH^2,
\end{align}
where $\wh{h}^n$ is the piecewise affine interpolant of $h^n$.
Since $I\in [0,\infty)$, we can assume that $\wh h^n$
is bounded in $\cH$ and is weakly converging to a limit $\overline{h}\in \cH$, which we see is equal to $h$.
By weak lower semicontinuity of the $L^2$ norm we obtain $\liminf_{n\to\infty}
C_n(\wt\bomega^{n},\bomega^n)^2\geq\|h\|_\cH^2$.

\emph{Case 1b.}
Now suppose that $\wt \bomega = T_h \bomega$, where $h = \pi_1 \wt \bomega - \pi_1 \bomega \notin \cH$. Then $C_\cH(\bomega, \wt \bomega) = \infty$. Supposing again that $I\coloneqq \liminf_{n\to \infty}C_n(\bomega^n, \wt \bomega^n) < \infty$, following the same argument as above leads to $h \in \cH$, which is a contradiction. Thus, $ \liminf_{n\to\infty} C_n(\wt\bomega^{n},\bomega^n)^2 = \infty$.

\emph{Case 2.} 
We now assume that $h \coloneqq \pi_1\wt \bomega - \pi_1\bomega\in\cH$ but $\wt \bomega\neq T_h\bomega$; i.e.\ 
$\wt \bomega$ is not a shift of $\bomega$.
We show that $C_{n}(\wt\bomega^n,\bomega^n) \to\infty$. Define $\wh \bomega = T_h\bomega$ such that, by assumption, $\wh\bomega\neq \wt\bomega$ but $\pi_1\wh \bomega = \pi_1\wt\bomega$. Therefore, there exists $\theta\in C([0,1], \R^{d_2})$ such that $\theta \not \equiv 0 $ and $\wt\bomega = \wh\bomega \btheta$ with $\btheta = (0,\theta)$.
By \Cref{prop:limsup}, we  find a sequence $\wh{\bomega}^{n}$ such that $\wh{\bomega}^{n}\to \wh{\bomega}$ and 
\begin{equation}
\label{eq:bou}
\limsup_{n\to\infty}C_{n}(\bomega^n,\wh{\bomega}^{n}) \leq C_{\cH}(\bomega,\wh\bomega)=\|h\|_\cH<\infty.
\end{equation}
Now, by the triangle inequality and the estimate $(a+b)^2\leq 2a^{2}+2b^{2}$, for any $a, b \in \R$, we see that
\[
C_{n}(\wh{\bomega}^{n},\wt{\bomega}^{n})^2\leq
2 C_{n}(\bomega^{n},\wh{\bomega}^{n})^2
+
2 C_{n}(\wt{\bomega}^{n},\bomega^{n})^2.
\]
By \eqref{eq:bou}, the first term on the right-hand side is bounded by $2\|h\|_\cH^2$. By \Cref{lem:vertical-cost}, we also have $\lim_{n\to \infty}C_{n}(\wh{\bomega}^{n},\wt{\bomega}^{n})^2=\infty$. Thus we conclude that $\lim_{n \to \infty} C_n(\wt \bomega^n, \bomega^n) = \infty$.
\end{proof}

Combining \cref{prop:limsup,prop:liminf}, we deduce the following $\Gamma$-convergence.
\begin{corollary}\label{cor:gamma-convergence}
	On $\bOmega_\bbG \times \bOmega_\bbG$ equipped with the uniform topology, we have the $\Gamma$-convergence $C_n \xrightarrow{\Gamma} C_\cH$.
\end{corollary}

\begin{proof}The liminf inequality follows directly from Proposition~\ref{prop:liminf}. The limsup inequality in this setting reads: For every pair $(\bomega,\wt \bomega)\in\bOmega_\bbG\times \bOmega_\bbG$,
we can find a sequence $(\bomega^n,\wt\bomega^n)$ converging to $(\bomega,\wt\bomega)$ such that 
\[
\limsup_{n\to \infty} C_n(\bomega^n,\wt\bomega^n)\leq 
C_n(\bomega,\wt\bomega).
\]
Proposition~\ref{prop:limsup} tells us that in fact we can take any sequence $\bomega^n$ converging to $\bomega$
and the sequence $\wt\bomega^n$ constructed via adding a suitable perturbation.
\end{proof}

\begin{remark}\label{rem:recoveryMap}
Let us note that Proposition~\ref{prop:limsup} is stronger than the standard $\limsup$ condition in $\Gamma$-convergence. In particular, we can choose the constant sequence $\bomega^n=\bomega$ such that the recovery sequence is obtained via a map $\Phi^n(\bomega,\wt\bomega) = (\bomega, \wt\bomega^n)$.
We will use this map $\Phi^n$ to construct sequences of transport plans $\wt\blambda^n$ that are recovery sequences for the family of  optimal transport problems associated with $C_n$; see Proposition \ref{prop:limsupTransProb} below.
\end{remark}

Having shown the $\Gamma$-convergence of the cost functions $C_n$, we can now deduce the $\Gamma$-convergence of the associated transport problems. For $n\in\N$, define the family of transport functionals $\mathrm{I}_n\colon\Pc(\bOmega_\bbG{\times} \bOmega_\bbG)\to [0,\infty]$ and $\mathrm{I}_\infty\colon\Pc(\bOmega_\bbG{\times} \bOmega_\bbG)\to [0,\infty]$ via
\[
\mathrm{I}_n(\blambda) = \int_{\bOmega_\bbG\times \bOmega_\bbG} C_n^2\di\blambda\quad\text{and}
\quad\mathrm{I}_\infty(\blambda) = \int_{\bOmega_\bbG\times \bOmega_\bbG} C_\cH^2\di\blambda, \quad \blambda \in \Pc(\bOmega_\bbG \times \bOmega_\bbG).
\]

\begin{proposition}
\label{prop:liminfTransProb}
    Let $(\blambda^n) \subset \Pc(\bOmega_\bbG\times \bOmega_\bbG)$ be a sequence of probability measures such that $\blambda^n\rightharpoonup\blambda$ in $\Pc(\bOmega_\bbG\times \bOmega_\bbG)$. Then
	\begin{equation*}
        \liminf_{n\to\infty} \mathrm{I}_n(\blambda^n) \geq  \mathrm{I}_\infty(\blambda).
    \end{equation*}
\end{proposition}

\begin{proof}
    By Skorokhod's representation theorem, there exists a probability space $(\Xi,\mathfrak{A},\mathbb{P})$ and random variables $\bY^n \colon \Xi \to \bOmega_\bbG\times \bOmega_\bbG$ and $\bY \colon \Xi \to \bOmega_\bbG\times \bOmega_\bbG$ such that $\blambda^n = \bY^n_\#\mathbb{P}$, $\blambda = \bY_\#\mathbb{P}$, and
    $\bY^n\to \bY$ $\mathbb{P}$-almost surely.
    We conclude that
    \begin{align*}
        \liminf_{n\to\infty} \int_{\bOmega_\bbG\times \bOmega_\bbG} C_n(\bomega,\overline\bomega)^2\di \blambda^n & = \liminf_{n\to\infty} \int_{\Xi} C_n(\bY^n)^2\di \mathbb{P}\\
        & \geq \int_\Xi \liminf_{n\to \infty} C_n(\bY^n)^2\di\mathbb{P} \geq \int_\Xi C_\cH(\bY)^2\di\mathbb{P} = \int_{\bOmega_\bbG\times \bOmega_\bbG} C_\cH(\bomega,\overline\bomega)^2\di \blambda,
    \end{align*}
    by Fatou's lemma and \Cref{prop:liminf}.
\end{proof}

\begin{remark}
	Given \Cref{prop:liminfTransProb}, we find a much more direct proof of \Cref{thm:t2-carnot-path}. Indeed, combining \Cref{lem:rel-entropy-bound} and \Cref{prop:liminfTransProb} yields the result immediately.
\end{remark}

\begin{proposition}
\label{prop:limsupTransProb}
Let $\blambda\in\Pc(\bOmega_\bbG\times \bOmega_\bbG)$. Then there exists a sequence
$(\wt \blambda^n) \subset \Pc(\bOmega_\bbG\times \bOmega_\bbG)$ such that $\wt\blambda^n\rightharpoonup\blambda$ and
\begin{equation}\label{eq:limsupPlan}
\limsup_{n\to\infty}\mathrm{I}_n(\wt\blambda^n)
\leq \mathrm{I}_\infty(\blambda).
\end{equation}
\end{proposition}

\begin{proof}We may assume that the right-hand side in \eqref{eq:limsupPlan} is finite as the inequality is trivially true otherwise.
In particular, we have $(\bomega,\overline\bomega)\mapsto C_\cH^2(\bomega,\overline\bomega)\in L^1(\blambda)$
and, for $\blambda$-almost every $(\bomega,\overline\bomega)\in \bOmega_\bbG\times \bOmega_\bbG$, we have that $\overline \bomega = T_h\bomega$
for $h=\pi_1(\ol\bomega^{-1}\bomega)\in\cH$.

Define $\wt\blambda^n = \Phi^n_\#\blambda\in \Pc(\bOmega_\bbG\times \bOmega_\bbG)$,
where $\Phi^n \colon \bOmega_\bbG\times \bOmega_\bbG\to \bOmega_\bbG\times \bOmega_\bbG$ maps $(\bomega,\wt\bomega)$ to  $(\bomega, \wt\bomega^n)$ as in \Cref{rem:recoveryMap}.
Then, for any $(\bomega, \wt \bomega) \in \bOmega_\bbG \times \bOmega_\bbG$, we have $\Phi^n(\bomega,\wt\bomega) \to (\bomega,\wt\bomega)$ and $\wt\blambda^n\rightharpoonup \blambda$ as $n \to \infty$. By \Cref{prop:limsup,prop:liminf},
\begin{equation}\label{eq:recovery-map-convergence}
	\lim_{n \to \infty}C_n(\Phi^n(\bomega,\wt\bomega)) = C_\cH(\bomega,\wt\bomega).
\end{equation}

Moreover, by \Cref{lem:lifted-CM-path}, 
we have $C_n(\bomega, \wt\bomega^n)\leq \|h\|_{\cH}=C_\cH(\bomega,\wt\bomega)$.
Using Fatou's lemma with integrable upper bound $C_\cH^2(\bomega,\wt\bomega)$ gives
\[
    \limsup_{n\to\infty}\int_{\bOmega_\bbG\times \bOmega_\bbG} C_n^2(\bomega,\wt\bomega)\di \wt\blambda^n
    = \limsup_{n\to\infty} \int_{\bOmega_\bbG\times \bOmega_\bbG} C_n^2(\Phi^n(\bomega,\wt\bomega))\di \blambda
    \leq \int_{\bOmega_\bbG\times \bOmega_\bbG} \limsup_{n\to\infty} C_n^2(\Phi^n(\bomega,\wt\bomega))\di\blambda.
\]
The assertion now follows from \eqref{eq:recovery-map-convergence}.
\end{proof}

\begin{corollary}
On $\Pc(\bOmega_\bbG \times \bOmega_\bbG)$ equipped with the weak topology, we have the $\Gamma$-convergence $\mathrm I_n \xrightarrow{\Gamma} \mathrm I_\infty$.
\end{corollary}

The following theorem is a version of the fundamental theorem of $\Gamma$-convergence in the present case.

\begin{theorem}\label{thm:GammaConvTransport}
	Let $\etab \in \Pc(\bOmega_\bbG)$. Then $\T_{C_n, 2}(\etab, \cdot) \xrightarrow{\Gamma} \T_{C_\cH, 2}(\etab, \cdot)$ with respect to the weak topology on $\Pc(\bOmega_\bbG)$. That is
	\begin{enumerate}[label = (\roman*)]
		\item For any $\bnu \in \Pc(\bOmega_\bbG)$ and any $(\bnu^n)\subset \Pc(\bOmega_\bbG)$ such that $\bnu^n \rightharpoonup \bnu$, $\liminf_{n \to \infty}\T_{C_n, 2}(\etab, \bnu^n) \geq \T_{C_\cH, 2}(\etab, \bnu)$; and
		\item For any $\bnu \in \Pc(\bOmega_\bbG)$, there exists a sequence $(\wt\bnu^n)\subset \Pc(\bOmega_\bbG)$ such that $\lim_{n\to\infty}\T_{C_n, 2}(\etab, \wt\bnu^n) = \T_{C_\cH, 2}(\etab, \bnu)$.
	\end{enumerate}
\end{theorem}

\begin{proof}
(i) Let $\bnu \in \Pc(\bOmega_\bbG)$ and $(\bnu^n)\subset \Pc(\Omega_\bbG)$ such that $\bnu^n \rightharpoonup \bnu$, and let $(\blambda^n) \subset \Pc(\bOmega_\bbG\times \bOmega_\bbG)$ be a sequence of optimal transport plans for $\T_{C_n, 2}(\etab, \wt\bnu^n)$. We can assume that $\blambda^n$ converges weakly to a limit $\blambda\in\Pc(\bOmega_\bbG\times \bOmega_\bbG)$ since its marginals are tight by Prokhorov's theorem; see \cite[Lemma 5.2.2]{AGS}. 
The limit $\blambda$ has marginals $\etab$ and $\bnu$ and is hence an admissible transport plan for $\T_{C_\cH, 2}(\etab, \bnu)$. Using Proposition \ref{prop:liminfTransProb}, we get the chain of inequalities
\[
	\T_{C_\cH, 2}(\etab, \bnu) \leq \int_{\bOmega_\bbG\times \bOmega_\bbG} C_\cH^2(\bomega, \ol\bomega) \di\blambda
	\leq \liminf_{n\to\infty}\int_{\bOmega_\bbG\times \bOmega_\bbG} C_n^2 (\bomega, \ol\bomega)\di\blambda^n  = \liminf_{n\to\infty}\T_{C_n, 2}(\etab, \wt\bnu^n).
\]

(ii) Now let $\blambda\in\Pc(\bOmega_\bbG\times \bOmega_\bbG)$ be an optimal transport plan for $\T_{C_\cH, 2}(\etab, \bnu)$ (note that $C_\cH$ is lower semi-continuous; see \cref{lem:cost-mb}). 
Let the sequence of transport plans $(\wt\blambda^n) \subset \Pc(\bOmega_\bbG \times \bOmega_\bbG)$ be given as in \Cref{prop:limsupTransProb}, and define $\wt\bnu^n$ as the second marginal of $\wt\blambda^n$, for $n \in \N$. The first marginal of $\wt\blambda^n$ is fixed to $\etab$ for all $n \in \N$.
Thus $\wt\bnu^n\rightharpoonup \bnu$. Moreover, by \cref{prop:limsupTransProb} and the optimality of $\blambda$,
\[
	\limsup_{n\to\infty} \int_{\bOmega_\bbG\times \bOmega_\bbG} C_n^2(\bomega,\overline\bomega)\di \wt\blambda^n \leq \T_{C_\cH, 2}^2(\etab, \bnu).
\]
On the other hand, if $(\ol\blambda^n) \subset \Pc(\bOmega_\bbG\times \bOmega_\bbG)$ is a sequence of 
optimal transport plans for $\T_{C_n, 2}(\etab, \wt\bnu^n)$, we can assume that $\ol\blambda^n\rightharpoonup \ol\blambda$, where the limit  
$\ol\blambda$ has marginals $\etab$ and $\bnu$. 
By Proposition \ref{prop:liminfTransProb}, we get
\[
	\T_{C_\cH, 2}^2(\etab, \bnu)\leq 
	\int_{\bOmega_\bbG\times \bOmega_\bbG} C_\cH^2(\bomega,\overline\bomega)\di \ol\blambda
	\leq\liminf_{n\to\infty}\T_{C_n, 2}^2(\etab, \wt\bnu^n).
\]
Combining both estimates proves the claim.
\end{proof}

\begin{remark}
In general, we cannot rule out that there exists a sequence $\bnu^n$ converging to some limit $\bnu$ such that $\lim_{n\to \infty} \T_{C_n,2}(\etab, \bnu^n) > \T_{C_\cH,2}(\etab, \bnu)$. The crucial point in Theorem~\ref{thm:GammaConvTransport} is that 
the sequence $\wt\bnu^n$ is a special sequence
constructed via the push-forward of the recovery map
$\Phi^n$. It is an interesting question whether the following stronger result holds: Let $\delta_{\Pi(\etab, \bnu^n)}$ denote the convex indicator function (taking values in $\{0,\infty\}$) for the set of admissible plans, i.e.\ $\delta_{\Pi(\etab, \bnu^n)}(\blambda) = 0$ if and only if $\blambda \in \Pi(\etab, \bnu^n)$,
and let $\bnu \in \Pc(\bOmega)$, $(\bnu^n) \subset \Pc(\bOmega_\bbG)$ such that $\bnu^n \rightharpoonup \bnu$ and $\sup_{n \in \N}H(\bnu^n \| \etab) < \infty$. Do we have the $\Gamma$-convergence $\mathrm{I}_n + \delta_{\Pi(\etab, \bnu^n)}\stackrel{\Gamma}
        {\rightarrow}\mathrm{I}_\infty + \delta_{\Pi(\etab,\bnu)}$?
        This property would imply that $\T_{C_n,2}(\etab, \bnu^n)\to \T_{C_\cH,2}(\etab, \bnu)$ for \textit{every} converging sequence $\bnu^n$ with finite relative entropy.
\end{remark}

\section{Beyond step-\texorpdfstring{$2$}{2} Carnot groups}\label{sec:outlook}

Parts of this work are valid in the generality of general Carnot groups (see, e.g.\ \cite{BoLaUg2007SLGP}). However, Carnot groups for which the log-Sobolev inequality is known are the Heisenberg group and more general H-type groups, which are examples of step-$2$ Carnot groups, as discussed in \cref{sec:carnot-step-2}. This explains our focus on step-$2$. Nevertheless, \cref{thm:gigli-ledoux} holds for general Carnot groups with no restriction on the step of the group, and the proof remains unchanged, given the appropriate definitions. Similarly, \cref{lem:t2-dilation}, \cref{prop:t2-tensorisation}, and \cref{lem:rel-entropy-bound} carry over without change to the general Carnot group setting. Finally, \cref{thm:t2-carnot-path} also holds for $\bbG = \mathbb F^{d_1, N}$, i.e.\ for step-$N$ free Carnot groups, under additional regularity assumptions for $N > 2$. On the space of $p$-variation paths, for any $p$ such that the shift by an absolutely continuous path is well defined, the proof of \cref{thm:t2-carnot-path} remains valid; see also \cref{rem:translation-regularity}.

\subsection*{Acknowledgement}
PF, HK, and ML acknowledge funding by the Deutsche Forschungsgemeinschaft (DFG, German Research Foundation) -- CRC/TRR 388 ``Rough Analysis, Stochastic Dynamics and Related Fields'' -- Project ID 516748464 within sub-project A02. This project also funded research stays of BR, during which part of this work was completed.
The research of VL has been partially funded by Deutsche Forschungsgemeinschaft (DFG, German Research Foundation) -- SFB1294/1-318763901.

\bibliographystyle{alpha}
\bibliography{references.bib}

\end{document}